\documentclass[10pt,leqno,twoside]{amsart}

\usepackage{tikz, subfigure, xcolor} 
\usetikzlibrary{snakes}
\usepackage{amsmath}
\usepackage{amsthm}
\usepackage{amssymb}
\usepackage{amsfonts}
\usepackage{cite}
\usepackage{dsfont}
\usepackage{hyperref}
\usepackage{doi,comment}
\newtheorem{theorem}{Theorem}[section]
\newtheorem{lemma}[theorem]{Lemma}
\newtheorem{proposition}[theorem]{Proposition}

\theoremstyle{definition}

\newtheorem{remark}[theorem]{Remark}

\numberwithin{equation}{section}

\renewcommand{\div}{\mathrm{div}\,}    

\providecommand{\norm}[1]{\lVert#1\rVert} 
\providecommand{\abs}[1]{\lvert#1\rvert} 

\DeclareMathOperator{\dom}{dom}

\newcommand{\R}{\mathbb{R}}

\newcommand{\N}{\mathbb{N}}

\newcommand{\PP}{\mathbb{P}}

\newcommand{\fa}{\mathfrak{a}}
\newcommand{\fb}{\mathfrak{b}}

\usepackage{eucal}

\newcommand\restr[2]{{
  \left.\kern-\nulldelimiterspace 
  #1 
  \vphantom{\big|} 
  \right|_{#2} 
  }}

\newcommand{\IP}{\mathbb{P}} 

\newcommand{\IC}{\mathbb{C}} 
\DeclareMathOperator{\divergence}{div}
\newcommand{\e}{\mathrm{e}} 
\newcommand{\eps}{\varepsilon} 
\renewcommand{\d}{\mathrm{d}} 

\title{Nematic Liquid Crystals in Lipschitz domains}

\keywords{Nematic Liquid Crystals, Ericksen--Leslie model, Parabolic equations in Lipschitz domains\\
This work was partly supported by the DFG International Research Training Group IRTG 1529.
The first and third authors are supported by IRTG 1529 at TU Darmstadt.}

\author{Anupam~Pal Choudhury} 

\address{Department of Mathematics,
TU Darmstadt, Schlossgartenstr. 7, 64289 Darmstadt, Germany}
\email{anupampcmath@gmail.com}
\email{hussein@mathematik.tu-darmstadt.de}
\email{tolksdorf@mathematik.tu-darmstadt.de}

\author{Amru Hussein}

\author{Patrick Tolksdorf} 

\begin{document}

\begin{abstract}
\noindent
We consider the simplified Ericksen--Leslie model in three dimensional bounded Lipschitz domains. Applying a semilinear approach, we prove local and global well-posedness (assuming a smallness condition on the initial data) in critical spaces for initial data in $L^3_{\sigma}$ for the fluid and $W^{1,3}$ for the director field. The analysis of such models, so far, has been restricted to domains with smooth boundaries.
\end{abstract}

\maketitle

\section{Introduction}
In this article, we establish a well-posedness theory for the isothermal simplified Ericksen--Leslie model in critical spaces on a bounded Lipschitz domain $\Omega \subset \R^3$. This model describes the flow of nematic liquid crystals and is given by the following system of equations
\begin{align}\label{eq:LCD}
\left\{
\begin{aligned}
\partial_{t} u + (u\cdot \nabla) u -\nu \Delta u + \nabla \pi &= - \lambda \text{div}([\nabla d]^{\top} \nabla d) \quad &\text{in} \ (0,T)\times \Omega,  \\
\partial_{t} d+(u \cdot \nabla) d &= \gamma (\Delta d+ \vert \nabla d \vert^{2} d) \quad &\text{in} \ (0,T)\times \Omega, \\
\div u &= 0 \quad &\text{in} \ (0,T)\times \Omega,\\
\end{aligned}\right.
\end{align}
with initial data $u(0)=a$ and $d(0)=b$. 
Here $u: (0 , T) \times \Omega \rightarrow \mathbb{R}^{3} $ denotes the velocity field of the fluid, $\pi: (0 , T) \times \Omega \rightarrow \mathbb{R} $ the pressure, and
$d: (0 , T) \times \Omega \rightarrow \mathbb{R}^{3} $ denotes the molecular orientation of the liquid crystal at the macroscopic level (we shall also refer to this as the director field). This physical interpretation of $d$ imposes the condition
\begin{equation}
\vert d \vert=1 \quad \text{in} \quad (0,T)\times \Omega.
\label{abs}
\end{equation}
We shall therefore further assume that $\vert d(0) \vert= \vert b \vert=1 \ \text{in}\ \Omega $.
The constant $\nu > 0$ represents the viscosity, the constant $\gamma > 0 $ represents the microscopic elastic relaxation time for the molecular orientation field $d$, and the constant $\lambda>0 $ encodes the  competition between the kinetic and potential energies. 
Without loss of generality, we shall restrict ourselves to the case $\lambda=\gamma=\nu=1$. This system is complemented with suitable boundary conditions for $u$ and $d$. The velocity field will always be assumed to satisfy no-slip boundary conditions
\begin{align*}
 u = 0 \quad \text{on} \quad (0 , T) \times \partial \Omega,
\end{align*}
and the director field either satisfies homogeneous Neumann boundary conditions 
\begin{align}\label{bc:Neumann}
 \partial_n d = 0  \quad \text{on} \quad (0 , T) \times \partial \Omega,
\end{align}
or it is assumed that the alignment of $d$ on the boundary is prescribed by a constant unit vector $e \in \mathbb{S}^2$, i.e.,
\begin{align}\label{bc:Dirichlet}
\quad d = e \quad \text{on} \quad (0 , T) \times \partial \Omega.
\end{align}
Here, $\partial_n d$ denotes the normal derivative of $d$ to $\partial \Omega$. Notice that both these types of boundary conditions for $d$ are physically relevant and have been investigated in smooth domains~\cite{Li_Wang},~\cite{Hieber_et_al},~\cite{Li}, and~\cite{Lin_Lin_Wang}.\par
After the continuum theory of liquid crystals was developed by Ericksen~\cite{Ericksen} and Leslie~\cite{Leslie} in the 1960's, a first simplified model (which is a slightly modified version of~\eqref{eq:LCD}) was considered by Lin and Liu~\cite{Lin_Liu} in 1995. In the case of bounded and smooth domains the above mentioned system was considered by Li~\cite{Li} subject to Dirichlet boundary conditions for $d$; subject to Neumann boundary conditions it was investigated by Li and Wang~\cite{Li_Wang} and by Hieber, Nesensohn, Pr\"uss, and Schade~\cite{Hieber_et_al}. While the two latter treatments rely both on maximal regularity estimates for the Stokes operator and the Neumann Laplacian, they differ in their underlying philosophy. Namely, Li and Wang treat it as a semilinear problem and Hieber \textit{et al.\@} advertise the quasilinear approach. Detailed information on liquid crystals including their history and further references can be found in the books by Sonnet and Virga~\cite{Virga2012} and Virga~\cite{Virga1994}. Recent developments are discussed by Hieber and Pr\"uss in the survey~\cite{HieberPruess2016}. \par
In this work, we shall view the simplified Ericksen--Leslie model as a semilinear equation and treat it by the semigroup method presented for example by Giga~\cite{Giga}, Giga and Miyakawa~\cite{GigaMiyakawa1985}, and Kato~\cite{Kato}. For instance in the case of Neumann boundary conditions for $d$, this means that all nonlinear terms are considered as a ``right-hand side'' and that we shall construct mild solutions
\begin{align*}
 u(t) &= e^{- t A} a - \int_0^t e^{- (t - s) A} \IP \big\{ (u (s) \cdot \nabla) u (s) + \text{div}([\nabla d (s)]^{\top} \nabla d (s)) \big\} \; \d s , \\
 d(t) &= e^{- t B} b - \int_0^t e^{- (t - s) B} \big\{ (u(s) \cdot \nabla) d(s) - \lvert \nabla d(s) \rvert^2 d(s) \big\} \; \d s
\end{align*}
by virtue of an iteration scheme. Here, the ``fluid equation'', i.e., the first equation of~\eqref{eq:LCD}, is projected onto the solenoidal vector fields by using the Helmholtz projection $\IP$ and $A$ denotes the Stokes operator; $- B$ denotes the Neumann Laplacian. As the underlying domain is only Lipschitz, there are profound constraints concerning the regularity of the involved operators. For example the Helmholtz projection on $L^p$ as well as the Stokes semigroup on $L^p_{\sigma}$ exist only for $3 / 2 - \eps < p < 3 + \eps$ and some $\eps = \eps(\Omega) > 0$, see Fabes, Mendez, and Mitrea~\cite{Fabes_Mendez_Mitrea}, Shen~\cite{Shen}, and Deuring~\cite{Deuring}. Another point is, that one cannot expect the domains of the operators $A$ and $B$ to embed into $W^{2 , p}$ for any $p > 1$, as this property is in general wrong for the Laplacian, see Dahlberg~\cite{Dahlberg} and Jerison and Kenig~\cite{Jerison-Kenig}. Firstly, this shows that one cannot expect an $L^p$-result for $p \geq 3 + \eps$ and secondly, this directly leads to problems of how to interpret the mild solutions above, as the Stokes semigroup is applied to two derivatives of $d$. \par
To circumvent this problem, we shall write $(u(s) \cdot \nabla ) u(s)$ as $\text{div} (u(s) \otimes u(s))$ and then consider
\begin{align*}
 e^{- (t - s) A} \IP \text{div}
\end{align*}
as one composite operator on $L^p$. That this is well-defined for $3 / 2 - \eps < p < 3 + \eps$ and a bounded Lipschitz domain $\Omega$, was proven by the third author in~\cite{Tolksdorf}. To prove convergence of the iteration scheme, it will be important that $u(s) \otimes u(s)$ and $[\nabla d(s)]^{\top} \nabla d(s)$ exhibit the same decay rate in the time variable, because in this case, both integrands in the mild formulation of $u$ behave similarly with respect to the time variable. Since the constructed solutions $u$ and $d$ will be perturbations of the solutions $e^{- t A} a$ and $e^{- t B} b$ to the linearised equations, we need to impose conditions on $b$ such that $\nabla e^{- t B} b$ has the same time decay as
\begin{align*}
 \| e^{- t A} a \|_{L^q_{\sigma} (\Omega)} \leq C t^{- \frac{3}{2} (\frac{1}{p} - \frac{1}{q})} \| a \|_{L^p_{\sigma} (\Omega)} \qquad (\tfrac{3}{2} - \eps < p \leq q < 3 + \eps).
\end{align*}
Since $B$ satisfies the square root property $\| \nabla f \|_{L^p} \simeq \| B^{1 / 2} f \|_{L^p}$ for $\frac{3}{2} - \eps < p < 3 + \eps$ and $f \in \dom(B^{1 / 2})$, one has the same time decay
\begin{align*}
\| \nabla e^{- t B} b \|_{L^q (\Omega)^{3 \times 3}} \leq C t^{- \frac{3}{2} (\frac{1}{p} - \frac{1}{q})} \| B^{1 / 2} b \|_{L^p (\Omega)^3} \qquad (\tfrac{3}{2} - \eps < p \leq q < 3 + \eps),
\end{align*}
whenever $b \in \dom(B^{1 / 2}) = W^{1 , p} (\Omega)^3$. \par
This leads us to an informal formulation of our main results. In Theorems~\ref{thm:main} and~\ref{Dirichlet-main}, 
we prove local existence of mild solutions to the simplified Ericksen--Leslie model for initial data $a \in L^p_{\sigma} (\Omega)$ and $b \in W^{1 , p} (\Omega)^3$ with $\lvert b \rvert = 1$ and every $3 \leq p < 3 + \eps$ for some $\eps > 0$. If $a$ and $\nabla b$ are sufficiently small in $L^p$, then the solutions are global. Especially, the solutions satisfy
\begin{align*}
 u \in BC ([0 , T) ; L^p_{\sigma} (\Omega)), \quad \nabla d \in BC ([0 , T) ; L^p(\Omega)^{3 \times 3}), \quad \text{and} \quad d \in BC ([0 , T) ; L^{\infty} (\Omega)^3)
\end{align*}
(actually, $d$ satisfies $\lvert d (t) \rvert = 1$ for all times). In the case $p = 3$, the norms of the spaces above are invariant under the natural scaling of the equation, i.e., if $u$, $d$, and $\pi$ are solutions to~\eqref{eq:LCD} and $\alpha > 0$, then so are
\begin{align*}
 u_{\alpha} (t , x) := \alpha u(\alpha^2 t , \alpha x), \quad d_{\alpha} (t , x) := d (\alpha^2 t , \alpha x), \quad \text{and} \quad \pi_{\alpha} (t , x) := \alpha^2 \pi (\alpha^2 t , \alpha x)
\end{align*}
(on a dilated domain and time interval). Thus, we establish a well-posedness theory for the simplified Ericksen--Leslie model in critical spaces. Furthermore, we prove that under certain conditions the mild solutions are unique. Having these solutions at hand, we proceed by regarding the nonlinearities as ``right-hand sides'' and use the theory of maximal regularity to prove additional regularity properties of the solutions. Finally, we would like to stress, that this is the first time that a well-posedness theory for the simplified Ericksen--Leslie model is established on a bounded Lipschitz domain and that we prove existence results for certain initial data spaces that are even unknown in the smooth case. To the best of our knowledge, well-posedness results in critical spaces have been obtained only for the full space $\R^3$, see~\cite{Hineman2013, Hineman2014, LinDing2012}. \par
The article is organised as follows. In Section~\ref{Sec: Preliminaries}, we introduce some of the basic tools and notations. The main results are stated in Section~\ref{Sec: Main results} and the iteration scheme (for Neumann boundary conditions for $d$) is performed in~\ref{Sec: Proof of Mild solvability}. In Section~\ref{Sec: Regularity} we prove regularity of the corresponding solutions and then, in Section~\ref{Sec: Dirichlet boundary conditions}, we outline the changes of the proof for Dirichlet boundary conditions. We close the article in Section~\ref{Sec: Comparison} with a comparison of our results in the smooth setting with previously known results.

\section{Preliminaries}
\label{Sec: Preliminaries}
In this section, we collect some preliminary results which shall be used time and again in the rest of the article. \par
For the whole article, $\Omega \subset \R^3$ will be a bounded Lipschitz domain, by which we mean that the boundary can locally be described by the graph of a Lipschitz continuous function. The space dimension of the underlying Euclidean space is always fixed to three. Integration will always be performed with respect to the Lebesgue measure. For two vectors $x , y \in \R^3$ we denote by $x \otimes y$ the matrix that arises by carrying out the matrix multiplication $x y^{\top}$, where the superscript $^{\top}$ denotes the transpose of a matrix. For a linear operator $C$ defined on a Banach space $X$, we denote its domain by $\dom(C) \subset X$ and its range by $\mathrm{Rg} (C)$. \par
Define the space of all solenoidal, smooth, and compactly supported vector fields by $C_{c , \sigma}^{\infty} (\Omega)$. Then, for $1 < q < \infty$, we denote by
\begin{align*}
 L^q_{\sigma} (\Omega) := \overline{C_{c , \sigma}^{\infty} (\Omega)}^{L^q} \quad \text{and} \quad W^{1 , q}_{0 , \sigma} (\Omega) := \overline{C_{c , \sigma}^{\infty} (\Omega)}^{W^{1 , q}}
\end{align*}
the $L^q$-space and the first-order Sobolev space of solenoidal vector fields. Moreover, for $q^{\prime}$ being the H\"older conjugate exponent to $q$, we define $W^{-1 , q}_{\sigma} (\Omega) := (W^{1 , q^{\prime}}_{0 , \sigma} (\Omega))^*$, where the $^*$ indicates that the antidual space was taken. Finally, for a Banach space $X$ and an interval $I \subset \R$, we denote by $BC (I ; X)$ the space of all bounded and continuous functions endowed with the supremum norm and we will denote the space of all average-free $L^q$-functions by
\begin{align*}
 L^q_0 (\Omega)^3 := \Big\{ d \in L^q(\Omega)^3 \mid \int_{\Omega} d \; \d x = 0 \Big\}.
\end{align*}
Recall that it was proven by Fabes, Mendez, and Mitrea in~\cite{Fabes_Mendez_Mitrea} that for each Lipschitz domain $\Omega \subset \R^3$, there exists $\eps > 0$ such that the Helmholtz projection $\IP$ from $L^q(\Omega)^3$ onto $L^q_{\sigma} (\Omega)$ is a bounded operator, whenever $\lvert 1 / q - 1 / 2 \rvert < 1 / 6 + \eps$. In this case, one can canonically identify $L^q_{\sigma} (\Omega)$ with the antidual space $(L^{q^{\prime}}_{\sigma} (\Omega))^*$. This means, that for every $f \in (L^{q^{\prime}}_{\sigma} (\Omega))^*$ there exists a unique $g \in L^q_{\sigma} (\Omega)$ such that
\begin{align*}
 \langle f , u \rangle_{[L^{q^{\prime}}]^*, L^{q^{\prime}}} = \int_{\Omega} g (x) \cdot \overline{u (x)} \; \d x \quad (u \in L^{q^{\prime}}_{\sigma} (\Omega)),
\end{align*}
and we denote the corresponding isomorphism by $\Phi : (L^{q^{\prime}}_{\sigma} (\Omega))^* \to L^q_{\sigma} (\Omega)$. \par
The Stokes operator and the Neumann Laplacian are defined by means of Kato's form method as follows. Define the sesquilinear forms
\begin{align*}
 \fa : W^{1 , 2}_{0 , \sigma} (\Omega) \times W^{1 , 2}_{0 , \sigma} (\Omega) \to \IC, \quad &(u , v) \mapsto \int_{\Omega} \nabla u \cdot \overline{\nabla v} \; \d x, \\
 \fb : W^{1 , 2} (\Omega)^3 \times W^{1 , 2} (\Omega)^3 \to \IC, \quad &(u , v) \mapsto \int_{\Omega} \nabla u \cdot \overline{\nabla v} \; \d x,
\end{align*}
and let the Stokes operator $A_2$ be the $L^2_{\sigma} (\Omega)$-realisation of $\fa$ and the negative Neumann Laplacian $B_2$ be the $L^2 (\Omega)^3$-realisation of $\fb$. For $1 < q < \infty$, the Stokes operator $A_q$ is either defined as the part of $A_2$ in $L^q_{\sigma} (\Omega)$ (if $q > 2$) or as the closure of $A_2$ in $L^q_{\sigma} (\Omega)$ (if $q < 2$) whenever the closure exists. In the same way, we define the Neumann Laplacian $B_q$ on $L^q (\Omega)^3$. Note that the respective operators are closable in the case $q < 2$ if and only if $A_{q^{\prime}}$ (or $B_{q^{\prime}}$) is densely defined, for a proof, see, e.g.,~\cite[Lem.~2.8]{Tolksdorf_Higher-order}. Moreover, if one of these conditions apply, then
\begin{align*}
 \langle A_q u , v \rangle_{L^q_{\sigma} , L^{q^{\prime}}_{\sigma}} = \langle u , A_{q^{\prime}} v \rangle_{L^q_{\sigma} , L^{q^{\prime}}_{\sigma}} \quad (u \in \dom(A_q) , v \in \dom(A_{q^{\prime}}))
\end{align*}
and
\begin{align*}
 \langle B_q u , v \rangle_{L^q , L^{q^{\prime}}} = \langle u , B_{q^{\prime}} v \rangle_{L^q , L^{q^{\prime}}} \quad (u \in \dom(B_q) , v \in \dom(B_{q^{\prime}})).
\end{align*}

The following result of Shen shows that $-A_q$ generates an exponentially stable analytic semigroup on $L^q_{\sigma} (\Omega)$ whenever $\lvert 1 / q - 1 / 2 \rvert < 1 / 6 + \eps$, in particular, this implies that the Stokes operator $A_q$ is closed and densely defined for $q$ is this range.

\begin{proposition}[see~\cite{Shen}]
\label{Prop: Shen's theorem}
For every bounded Lipschitz domain $\Omega \subset \R^3$ there exists $\eps > 0$ such that for all $q$ satisfying $\lvert 1 / q - 1 / 2 \rvert < 1 / 6 + \eps$, the operator $- A_q$ generates an exponentially stable analytic semigroup on $L^q_{\sigma} (\Omega)$.
\end{proposition}

The heat semigroup generated by the Neumann Laplacian has the following properties on bounded Lipschitz domains.

\begin{proposition}\label{prop:Linfty}
For all $q \in (1 , \infty)$ the operator $- B_q$ is the generator of a bounded analytic contraction semigroup $(e^{-t B_{q}})_{t \geq 0}$ on $L^q(\Omega)^3$ and for $q=\infty$, $(e^{-t B_q})_{t \geq 0}$ is contractive, i.e.,
\begin{align*}
\norm{e^{-t B_{\infty}}d}_{L^{\infty}(\Omega)^3} \leq \norm{d}_{L^{\infty}(\Omega)^3} \quad \hbox{for all } d\in L^{\infty}(\Omega)^3.
\end{align*}
\end{proposition}
\begin{proof}
By~\cite[Thm.~1.3.9]{Davies}, $B_{2}$ satisfies the so-called Beurling--Deny conditions. In this case,~\cite[Thm.~1.3.3]{Davies} implies that $(e^{-t B_{q}})_{t \geq 0}$ is a semigroup of contractions on $L^q (\Omega)^3$ for $1 \leq q \leq \infty$. The analyticity for $q \in (1 , \infty)$ follows from~\cite[Thm.~1.4.2]{Davies}.
\end{proof}

Since the operator domains $\dom(A_q)$ and $\dom(B_q)$ are nested for decreasing $q$, the corresponding semigroups define a consistent family of operators, $e^{- t A_q}|_{L^p_{\sigma} (\Omega)} = e^{- t A_p}$ and $e^{- t B_q}|_{L^p (\Omega)^3} = e^{- t B_p}$ for $p > q$. Thus, if no ambiguity is expected, we will follow the standard convention and skip the subscript $_q$ henceforth and simply write $A$ and $B$. The following result characterises the domains of the square roots of $A$ and $B$ defined above.

\begin{proposition}[see~\cite{Tolksdorf} and~\cite{Jerison-Kenig}]\label{prop:sqrt}
\label{Prop: Square root properties}
Let $\Omega\subset \R^3$ be a bounded Lipschitz domain. Then there exists an $\eps > 0$ such that for all $\lvert 1 / q - 1 / 2 \rvert < 1 / 6 + \eps$, 
\begin{itemize}
\item[(a)] one has with equivalent norms
$$\dom(A^{\frac{1}{2}})= W_{0 , \sigma}^{1,q}(\Omega),$$  
\item[(b)] one has 
$$\dom(B^{\frac{1}{2}})= W^{1,q}(\Omega)^3$$
and there exists a constant $C > 0$ such that
$$ C^{-1} \| \nabla u \|_{L^q (\Omega)^{3 \times 3}} \leq \| B^{\frac{1}{2}} u \|_{L^q(\Omega)^3} \leq C \| \nabla u \|_{L^q(\Omega)^{3 \times 3}}  \quad (u \in W^{1 , q} (\Omega)^3).$$  
\end{itemize}
\end{proposition}

In the following proposition, we recall and outline the proofs of $L^p$-$L^q$-type estimates for the Stokes and the heat semigroups.

\begin{proposition}\label{prop:smoothing}
Let $\Omega \subset \R^3$ be a bounded Lipschitz domain. Then there exists $\eps>0$ such that,
\begin{itemize}
\item[(a)] there exists $\omega>0$ and a constant $C > 0$ such that for $\frac{3}{2}-\eps< p \leq q < 3+\eps$ and $t>0$, 
\begin{align*}
\norm{e^{-t A} f}_{L^q_{\sigma}(\Omega)}  &\leq C e^{-\omega t} t^{-\tfrac{3}{2}\left(\tfrac{1}{p} - \tfrac{1}{q} \right)} \norm{f}_{L^p_{\sigma}(\Omega)}, \quad f\in L^p_{\sigma}(\Omega),\\
\norm{e^{-t A} \PP \div F}_{L^q_{\sigma}(\Omega)}  &\leq C e^{-\omega t} t^{-\tfrac{1}{2}-\tfrac{3}{2}\left(\tfrac{1}{p} - \tfrac{1}{q} \right)} \norm{F}_{L^p(\Omega)^{3 \times 3}}, \quad F\in L^p(\Omega)^{3 \times 3},
\end{align*}
where $e^{-t A} \PP \div$ is the $L^p$-extension of the respective operator defined a priori on $C_c^{\infty} (\Omega)^{3 \times 3}$.
\item[(b)] there exists $\omega>0$ and a constant $C > 0$ such that for all $1 < p \leq q \leq \infty$ with $p < \infty$ and $t>0$,
\begin{align*}
\norm{e^{-t B} f}_{L^q(\Omega)^3}  &\leq C e^{-\omega t} t^{-\tfrac{3}{2}\left(\tfrac{1}{p} - \tfrac{1}{q} \right)} \norm{f}_{L^p(\Omega)^3}, \quad f\in L^p_0(\Omega)^3,
\end{align*}
using the convention $\frac{1}{\infty} = 0$. Moreover, for all $\frac{3}{2} - \eps < p \leq q < 3 + \eps$ and for $t > 0$ it holds
\begin{align*}
\norm{\nabla e^{-t B} f}_{L^q(\Omega)^{3 \times 3}}  &\leq C e^{-\omega t} t^{-\tfrac{3}{2}\left(\tfrac{1}{p} - \tfrac{1}{q} \right)} \norm{\nabla f}_{L^p(\Omega)^{3 \times 3}}, \quad f\in L^p_0(\Omega)^3 \cap W^{1,p}(\Omega)^3, \\
\norm{\nabla e^{-t B}  f}_{L^q(\Omega)^{3 \times 3}}  &\leq C e^{-\omega t} t^{-\tfrac{1}{2}-\tfrac{3}{2}\left(\tfrac{1}{p} - \tfrac{1}{q} \right)} \norm{f}_{L^p(\Omega)^3}, \quad f\in L^p_{0}(\Omega)^3.
\end{align*}
\end{itemize}
\end{proposition}

\begin{proof}
The first estimate in (a) was proven in~\cite[Thm.~1.2]{Tolksdorf} but without an exponential decay factor. The exponential decay can be obtained by using the semigroup law $e^{- t A} = e^{- \frac{t}{2} A} e^{- \frac{t}{2} A}$ and then by using the exponential decay of the semigroup on $L^q_{\sigma} (\Omega)$ first, followed by the corresponding $L^p$-$L^q$-estimate from~\cite[Thm.~1.2]{Tolksdorf}. The second estimate in (a) is derived similarly by dualising the gradient estimates $t^{1 / 2} \| \nabla e^{- t A} f \|_{L^p_{\sigma} (\Omega)} \leq C \| f \|_{L^p_{\sigma} (\Omega)}$ and then by employing the semigroup law as above and the first estimate in (a). \par
To prove the first estimate in (b), notice that the heat kernel $k_t (x , y)$ of the heat semigroup $(e^{-t B})_{t \geq 0}$ admits the following estimate
\begin{align*}
 \lvert k_t (x , y) \rvert \leq C_1 \max \{ t^{- \frac{3}{2}} , 1 \} e^{- \frac{\lvert x - y \rvert^2}{C_2 t}}
\end{align*}
for some constants $C_1 , C_2 > 0$, see~\cite[Thm.~3.2.9]{Davies}. With this and Young's inequality, it follows
\begin{align*}
 \| e^{-t B} f \|_{L^q(\Omega)^3} \leq C_1 \max \{ t^{- \frac{3}{2}} , 1 \} \Big\| x \mapsto e^{- \frac{\lvert x \rvert^2}{C_2 t}} \Big\|_{L^r (\R^3)} \| f \|_{L^p (\Omega)^3},
\end{align*}
where $1 \leq r < \infty$ is such that
 $1 + \frac{1}{q} = \frac{1}{r} + \frac{1}{p}$.
This implies 
\begin{equation}
\| e^{-t B} f \|_{L^q(\Omega)^3} \leq C \max \{ t^{- \frac{3}{2}} , 1 \} t^{\frac{3}{2 r}} \| f \|_{L^p (\Omega)^3} \leq C \max \{ 1 , t^{\frac{3}{2}} \} t^{-\frac{3}{2}(\frac{1}{p}-\frac{1}{q})} \| f \|_{L^p (\Omega)^3}. 
\notag
\end{equation}
Now let $f \in L^{p}_{0}(\Omega) $. Using the estimate above and splitting $e^{-tB}=e^{-\frac{t}{2}B} e^{-\frac{t}{2}B} $ yields
\begin{equation}
\norm{e^{-t B} f}_{L^q(\Omega)^3} \leq C e^{-\frac{t}{2}\omega_{1}} \max \Big\{ 1 , \Big(\frac{t}{2} \Big)^{\frac{3}{2}} \Big\}\ t^{-\frac{3}{2}(\frac{1}{p}-\frac{1}{q})} \norm{f}_{L^p (\Omega)^3} ,
\notag
\end{equation}
for some constant $\omega_{1}>0 $. The exponential decay is a consequence of the fact that $f$ has average zero. It is then easy to see that for some constant $\omega > 0 $
\begin{equation}
\norm{e^{-t B} f}_{L^q(\Omega)^3} \leq C e^{-\omega t} t^{-\frac{3}{2}(\frac{1}{p}-\frac{1}{q})} \norm{f}_{L^p (\Omega)^3}.
\notag
\end{equation}
The second and third estimate in (b) follow from the first by using Proposition~\ref{Prop: Square root properties}.
\end{proof}

Another important notion that is needed for the proof of the main result as well as regularity considerations of the solutions to~\eqref{eq:LCD} is the one of maximal regularity. \par
Let $X$ be a Banach space and $C : \dom(C) \subset X \to X$ be a closed and densely defined operator such that $- C$ generates a bounded analytic semigroup. Fix $1 < s < \infty$ and $0 < T \leq \infty$ and consider for $f \in L^s (0 , T ; X)$ and $c$ in the real interpolation space $(X , \dom(C))_{1 - 1 / s , s}$ the abstract Cauchy problem
\begin{align}
\label{Eq: ACP}
\left\{ \begin{aligned}
 u^{\prime} (t) + C u (t) &= f(t) \qquad (0 < t < T), \\
 u(0) &= c.
\end{aligned} \right.
\end{align}
It is well-known~\cite[Prop.~3.1.16]{ABHN}, that~\eqref{Eq: ACP} admits a unique mild solution $u$ that satisfies
\begin{align*}
 u(t) = e^{- t C} c + \int_0^t e^{- (t - s) C} f(s) \; \d s \qquad (0 < t < T).
\end{align*}
We say that $C$ has maximal $L^s$-regularity if \textit{for every} $f \in L^s (0 , T ; X)$ and \textit{every} $c \in (X , \dom(C))_{1 - 1 / s , s}$, the corresponding mild solution $u$ is differentiable for almost every $t$, satisfies $u (t) \in \dom(C)$ for almost every $t$, and $u^{\prime} , C u \in L^s (0 , T ; X)$. If $T$ is finite or, if $T = \infty$ and $C$ is boundedly invertible, then maximal $L^s$-regularity is equivalent to the fact that the mild solution to~\eqref{Eq: ACP} lies in the maximal regularity class
\begin{align*}
 u \in W^{1 , s} (0 , T ; X) \cap L^s (0 , T ; \dom(C)).
\end{align*}
Let us summarise some well-known facts: Maximal $L^s$-regularity is independent of $s$, i.e., $C$ has maximal $L^s$-regularity for some $1 < s < \infty$ if and only if it has maximal $L^s$-regularity for every $1 < s < \infty$, cf.~\cite{DHP}. Because of this, we will henceforth only write maximal regularity instead of maximal $L^s$-regularity. Another well-known fact is that it suffices to prove maximal regularity in the special case $c = 0$, see, e.g., the discussion in~\cite[Sec.~2.2]{Tolksdorf_Dissertation}. \par
For the Stokes operator, maximal regularity was proven by Kunstmann and Weis in~\cite[Prop.~13]{Kunstmann_Weis}, see also~\cite[Thm.~5.2.24]{Tolksdorf_Dissertation}. In the case of the negative Neumann Laplacian, maximal regularity follows from Proposition~\ref{prop:Linfty} combined with a result of Lamberton~\cite[Cor.~1.1]{Lamberton}.

\begin{proposition}
\label{Prop: Maximal regularity of strong operators}
Let $\Omega \subset \R^3$ be a bounded Lipschitz domain and $1 < T \leq \infty$.
\begin{enumerate}
 \item[(a)] There exists $\eps > 0$ such that for every $\lvert 1 / q - 1 / 2 \rvert < 1 / 6 + \eps$ the Stokes operator on $L^q_{\sigma} (\Omega)$ has maximal regularity.
 \item[(b)] For every $1 < q < \infty$, the negative Neumann Laplacian on $L^q(\Omega)^3$ has maximal regularity.
\end{enumerate}
\end{proposition}

We close this section with a final remark concerning the results of this section on smooth domains.

\begin{remark}\label{rem:smooth}
If $\partial \Omega$ is smooth, then all of the results mentioned in this section are valid on the whole interval $q\in (1 , \infty)$, see~\cite{Jerison-Kenig} for the corresponding results for the Laplacian and~\cite{Giga_Analyticity},~\cite{Giga_fractional},~\cite{Giga}, and~\cite{Geissert_et_al} for the Stokes operator.
\end{remark}
%
%
%
%
%
\section{Main Result}
\label{Sec: Main results}

To begin with, we (formally) apply the Helmholtz projection $\PP$ to the first equation in~\eqref{eq:LCD} and consider the resulting system of equations
\begin{align}\label{eq:PLCD}
\left\{
\begin{array}{rll}
\partial_t u   + A u  & = - \PP (u \cdot \nabla) u - \PP \div ([\nabla d]^{\top} \nabla d), \quad &\text{ in } (0,T)\times  \Omega ,  \\
\partial_t d + B d  & = - (u \cdot \nabla) d + \abs{\nabla d}^2 d, \quad &\text{ in } (0,T) \times \Omega  , \\
\end{array}\right.
\end{align}
with initial conditions $u(0)=a$ and $d(0)=b$ on the space 
\begin{align*}
X= L_{\sigma}^q(\Omega) \times L^q(\Omega)^3.
\end{align*}

Our aim is to construct a \textit{mild solution} to~\eqref{eq:PLCD}, that is, a solution to the integral equations
\begin{align}
\begin{split}\label{eq:mild_solution}
u(t) &= e^{- t A} a - \int_0^t e^{- (t - s) A} \IP \text{div} \big\{ u (s) \otimes u (s) + [\nabla d (s)]^{\top} \nabla d (s) \big\} \; \d s , \\
d(t) &= e^{- t B} b - \int_0^t e^{- (t - s) B} \big\{ (u(s) \cdot \nabla) d(s) - \lvert \nabla d(s) \rvert^2 d(s) \big\} \; \d s,
\end{split}
\end{align}
and then to show that this solution preserves the condition $\abs{d}=1$ if $\abs{d(0)}=1$, and therefore~\eqref{eq:PLCD} turns out to be equivalent to~\eqref{eq:LCD}-\eqref{abs}. \\

For $0 < T \leq \infty$ and $3 \leq p < q$, the class of solutions considered is defined using
\begin{align*}
&S_q^u(T):= \Big\{ u\in C((0,T);L_{\sigma}^q(\Omega))\mid \sup_{0< s < T} e^{\frac{\omega s}{2}} s^{\tfrac{3}{2}\left(\tfrac{1}{p} - \tfrac{1}{q} \right)} \norm{u(s)}_{L^q_{\sigma}(\Omega)} < \infty \Big\}, \\
&S_q^d(T):= \Big\{ d \in C((0,T);W^{1,q}(\Omega)^3 )\mid \sup_{0< s < T} e^{\frac{\omega s}{2}} s^{\tfrac{3}{2}\left(\tfrac{1}{p} - \tfrac{1}{q} \right)} \norm{\nabla d (s)}_{L^q(\Omega)^{3 \times 3}} < \infty \Big\},
\end{align*}
where $\omega > 0$ is the minimum of the corresponding constants appearing in Proposition~\ref{prop:smoothing}.

\subsection{Neumann boundary condition for the director field}
Define for any $b\in L^1(\Omega)$, the average and the complementary mean value free part
\begin{align}\label{eq:split}
\overline{b} := \frac{1}{\abs{\Omega}}\int_{\Omega} b \; \d x \qquad \hbox{and} \qquad b_s := b-\overline{b}, \quad \hbox{where } \int_{\Omega} b_s \; \d x =0.
\end{align}

The main result reads as follows.

\begin{theorem}\label{thm:main}
Let $\Omega\subset \R^3$ be a bounded Lipschitz domain, then there exists $\eps > 0$ such that given initial conditions $a \in L^p_{\sigma}(\Omega)$ and $b\in W^{1,p}(\Omega)^3 \cap L^{\infty}(\Omega)^3$ where $3 \leq p < 3+\eps$, the following hold true for $q \in (p, 3+\eps)$.
\begin{itemize}
\item[(a)] There exists $T > 0$ depending on the initial data such that equation~\eqref{eq:PLCD} with Neumann boundary conditions~\eqref{bc:Neumann} for $d$ has a local mild solution $(u , d)$ satisfying 
\begin{align*}
\begin{split}
&u\in S_q^u(T) \cap BC([0,T);L_{\sigma}^p(\Omega)),\\
&d_s \in S_q^d(T) \cap BC([0,T);W^{1,p}(\Omega)^3) \cap BC([0,T);L^{\infty}(\Omega)^3),\qquad
\overline{d} \in BC([0,T);\R^3),
\end{split}
\end{align*}
where in the limit $s \to 0+$, one has
\begin{align*}
\norm{u(s) - a}_{L^p_{\sigma}(\Omega)} \to 0, \quad \norm{d(s)-b}_{L^{\infty}(\Omega)^3} \to 0, \quad \norm{\nabla [d (s) - b]}_{L^p(\Omega)^{3\times 3}} \to 0. 
\end{align*}
\item[(b)]
In the limit $s \to 0+$, the solutions satisfy  
\begin{align*}
s^{\tfrac{3}{2}\left(\tfrac{1}{p} - \tfrac{1}{q} \right)} \norm{u(s)}_{L^q_{\sigma}(\Omega)} \to 0 \quad \hbox{and} \quad s^{\tfrac{3}{2}\left(\tfrac{1}{p} - \tfrac{1}{q} \right)} \norm{\nabla d (s)}_{L^q(\Omega)^{3\times 3}} \to 0. 
\end{align*}
\item[(c)] If $a$ and $\nabla b$ are sufficiently small, then the solution exists globally in the class
\begin{align*}
&u\in S_q^u(\infty) \cap BC([0,\infty);L_{\sigma}^p(\Omega)),\\
&d_s \in S_q^d(\infty) \cap BC([0,\infty);W^{1,p}(\Omega)^3) \cap BC([0,\infty);L^{\infty}(\Omega)^3),\qquad
\overline{d} \in BC([0,\infty);\R^3).
\end{align*}
\item[(d)] 
The solution is unique in the class given in $(a)$ provided $p>3$, and in the case $p=3$, it is unique in the subset of this class satisfying in addition the limit conditions $(b)$. 

 \item[(e)] Equation~\eqref{eq:PLCD} subject to Neumann boundary conditions~\eqref{bc:Neumann} preserves the condition $\abs{d}=1$ if $\abs{d(0)}=\abs{b}=1$.
\end{itemize}
\end{theorem}

\begin{remark}\label{rem:kosdfgjhiodkwßeshjlfg}
	We note that the number $\eps > 0$ is minimum of the corresponding constants appearing in Section~\ref{Sec: Preliminaries}. 
	
	Furthermore, the smallness condition in Theorem~\ref{thm:main} $(c)$ can be made precise in the sense that
	there exists a constant $C>0$ depending only on $p$, $q$, and $\Omega$ such that if
	\begin{align*}
	\max\{ \kappa, \kappa^2 \} (1 + \norm{b}_{L^{\infty}(\Omega)^3})  
	< C, \quad \hbox{where} \quad \kappa:=\norm{a}_{L^p_{\sigma}(\Omega)} + \norm{\nabla b}_{L^p(\Omega)^{3\times 3}},
	\end{align*}
	then the solution exists globally.
\end{remark}

\begin{theorem}\label{Regularity}
For every $s \in (1 , 2)$, the solution in Theorem~\ref{thm:main} has the following additional regularity properties 
\begin{align*}
	&u \in W^{1 , s} (0 , T ; W^{-1 , \frac{p}{2}}_{\sigma} (\Omega)) \cap L^s (0 , T ;  W^{1 , \frac{p}{2}}_{0 , \sigma} (\Omega)), \\
	&d^{\prime} , B_{\frac{p}{2}} d  \in L^s (0 , T ; L^{\frac{p}{2}} (\Omega)^3).
\end{align*}
\end{theorem}
\subsection{Dirichlet boundary condition for the director field}
In case of Dirichlet boundary conditions~\eqref{bc:Dirichlet} for the director field, consider the new variable
\begin{align*}
\delta = d-e 
\end{align*}
with homogeneous Dirichlet boundary conditions.
Denoting by $B$ the negative Dirichlet Laplacian, we construct a \textit{mild solution} to the transformed equation~\eqref{model3}, that is, a solution to the integral equations
\begin{align}\label{eq:mildDirichlet}
\begin{split}
u(t) &= e^{- t A} a - \int_0^t e^{- (t - s) A} \IP \text{div} \big\{ u (s) \otimes u (s) + [\nabla \delta (s)]^{\top} \nabla \delta (s) \big\} \; \d s , \\
\delta(t) &= e^{- t B} \tilde{b} - \int_0^t e^{- (t - s) B} \big\{ (u(s) \cdot \nabla) \delta(s) - \lvert \nabla \delta(s) \rvert^2 (\delta(s) +e) \big\} \; \d s,
\end{split}
\end{align}
where $\tilde{b}=b-e$.

\begin{theorem}\label{Dirichlet-main}
	Let $\Omega\subset \R^3$ be a bounded Lipschitz domain, then there exists $\eps > 0$ such that given initial conditions $a \in L^p_{\sigma}(\Omega)$ and $b\in W^{1,p}(\Omega)^3 \cap L^{\infty}(\Omega)^3$ with $b=e$ on $\partial \Omega$ for some $e \in \mathbb{S}^2$ where $3 \leq p < 3+\eps$, the following hold true for $q \in (p, 3+\eps)$.
	\begin{itemize}
		\item[(a)] There exists $T > 0$ depending on the initial data such that equation~\eqref{eq:mildDirichlet} with Dirichlet boundary conditions~~\eqref{bc:Dirichlet} has a local mild solution $(u , \delta)$ satisfying 
		\begin{align*}
		\begin{split}
		&u\in S_q^u(T) \cap BC([0,T);L_{\sigma}^p(\Omega)),\\
		&\delta \in S_q^d(T) \cap BC([0,T);W_0^{1,p}(\Omega)^3) \cap BC([0,T);L^{\infty}(\Omega)^3), 
		\end{split}
		\end{align*}
		where in the limit $s \to 0+$, one has
		\begin{align*}
		\norm{u(s) - a}_{L^p_{\sigma}(\Omega)} \to 0, \quad \norm{\delta(s)- \tilde{b}}_{L^{\infty}(\Omega)^3} \to 0, \quad \norm{\nabla [\delta (s) - \tilde{b}]}_{L^p(\Omega)^{3\times 3}} \to 0. 
		\end{align*}
		\item[(b)]
		In the limit $s \to 0+$, the solutions satisfy  
		\begin{align*}
		s^{\tfrac{3}{2}\left(\tfrac{1}{p} - \tfrac{1}{q} \right)} \norm{u(s)}_{L^q_{\sigma}(\Omega)} \to 0 \quad \hbox{and} \quad s^{\tfrac{3}{2}\left(\tfrac{1}{p} - \tfrac{1}{q} \right)} \norm{\nabla \delta (s)}_{L^q(\Omega)^{3\times 3}} \to 0. 
		\end{align*}
		\item[(c)] If $a$ and $\nabla b$ are sufficiently small, then the solution exists globally in the class
		\begin{align*}
		&u\in S_q^u(\infty) \cap BC([0,\infty);L_{\sigma}^p(\Omega)),\\
		&\delta \in S_q^d(\infty) \cap BC([0,\infty);W_0^{1,p}(\Omega)^3) \cap BC([0,\infty);L^{\infty}(\Omega)^3) .
		\end{align*}
		\item[(d)] The solution is unique in the class given in $(a)$  provided $p>3$, and in the case $p=3$, it is unique in the subset of this class satisfying in addition the limit conditions $(b)$. 
		\item[(e)] Equation~\eqref{eq:PLCD} subject to Dirichlet boundary conditions~\eqref{bc:Dirichlet} preserves the condition $\abs{d}=1$ if $\abs{d(0)}=\abs{b}=1$.
	\end{itemize}
\end{theorem}
Concerning $\eps>0$ and the smallness condition in Theorem~\ref{Dirichlet-main} $(c)$, analogous statements to Remark~ \ref{rem:kosdfgjhiodkwßeshjlfg} hold.

\begin{theorem}\label{Dirichlet-regularity}
	For every $s \in (1 , 2)$, the solution in Theorem~\ref{Dirichlet-main} has the following additional regularity properties 
	\begin{align*}
	u & \in W^{1 , s} (0 , T ; W^{-1 , \frac{p}{2}}_{\sigma} (\Omega)) \cap L^s (0 , T ;  W^{1 , \frac{p}{2}}_{0 , \sigma} (\Omega)), \\
	\delta & \in W^{1 , s} (0 , T ; L^{\frac{p}{2}} (\Omega)^3) \cap L^s (0 , T ; \dom(B_{\frac{p}{2}})).
	\end{align*}
\end{theorem}

\section{Proof of Theorem~\ref{thm:main}}
\label{Sec: Proof of Mild solvability}
In this section, we prove first the existence and uniqueness of mild solutions in the case of Neumann boundary conditions for the director field $d$, and second using the existence and uniqueness, we prove that $\abs{d(t)}=1$ holds if $\abs{d(0)}=1$. The number $\eps > 0$ denotes the minimal $\eps$ appearing in Section~\ref{Sec: Preliminaries}.
\subsection{Existence and uniqueness}
\label{Subsec: Existence and uniqueness}
We note that the semigroup generated by $- B$ is not exponentially stable on $L^q(\Omega)^3$. However, it is exponentially stable on $L_0^q(\Omega)^3$, see Proposition~\ref{prop:smoothing}, the complementary subspace of which are the constant functions. 
So, in order to achieve the global well-posedness result a change of coordinates is useful to split the exponentially stable part of $- B$ from the constant part, compare~\cite{PSZ} for a far more general method. \par
Note that~\eqref{eq:split} defines bounded projections in all $L^p$-spaces, $p\in [1,\infty]$, defined by 
\begin{align*}
P_c d = \overline{d} \qquad \hbox{and} \qquad P_s d = d_s.
\end{align*}
Now, using~\eqref{eq:split} we can define the new variables
\begin{align*}
x= \overline{d} - \overline{b} \qquad \hbox{and} \qquad y = d_s,
\end{align*}
where $x(0)=0$ and $y(0)= b_s$. Since
\begin{align*}
\Delta x = 0, \quad \nabla x = 0, \quad \Delta y = \Delta d, \quad \nabla y = \nabla d
\end{align*}
and
\begin{align*}
P_c (u \cdot \nabla) y = \frac{1}{|\Omega|}\int_{\Omega} (u \cdot \nabla) y \ \d x = \frac{1}{|\Omega|} \bigg( \int_{\Omega} u \cdot \nabla y_k \ \d x \bigg)_{1 \leq k \leq 3} =  0 \quad \hbox{for } u\in L_{\sigma}^p(\Omega),\ y\in W^{1,p}(\Omega)^3,
\end{align*}
one obtains as a reformulation of~\eqref{eq:PLCD}
\begin{align*}
\left\{
\begin{array}{rll}
\partial_t u   + A u  & = - \PP ( u \cdot \nabla ) u - \PP \div ([\nabla y]^{\top} \nabla y ), \quad &\text{ in } \Omega \times (0,T),  \\
\partial_t y + B y  & = - (u \cdot \nabla) y + P_s \abs{\nabla y}^2 (x+y+\overline{b}), \quad &\text{ in } \Omega \times (0,T), \\
\partial_t x & =    P_c \abs{\nabla y}^2 (x+y+\overline{b}), \quad &\text{ in } \Omega \times (0,T),
\end{array}\right.
\end{align*}
which defines a system in the space
\begin{align*}
L^q_{\sigma}(\Omega) \times L^q_{0}(\Omega)^3 \times \R^3.
\end{align*}

The nonlinear terms are comprised, using the representation $(u \cdot \nabla) u= \div u \otimes u$ for $\div u =0$, by the notation
\begin{align*}
F_u(u, \nabla y) &= - \PP \div (u \otimes u + [\nabla y]^{\top}  \nabla y ), \\
F_y(u, \nabla y, y, x, \overline{b})  &= - (u \cdot \nabla) y + P_s \abs{\nabla y}^2 (x+y+\overline{b}), \\
F_x(\nabla y, y, x, \overline{b}) &=  P_c \abs{\nabla y}^2 (x+y+\overline{b}).
\end{align*}

Starting with the mild formulation of the problem, we can now define the iteration scheme as follows. For $j\in \N_0$, define
\begin{eqnarray*}
u_0 :=e^{-tA}a,& & u_{j+1}  := u_0 + \int_0^t e^{-(t-s)A} F_u(u_j(s), \nabla y_j(s)) \ \d s,\\
y_0 :=e^{-t B}b_s,& & y_{j+1}  := y_0 + \int_0^t e^{-(t-s)B} F_y(u_j(s), \nabla y_j(s), y_j(s), x_j(s), \overline{b}) \ \d s, \\
x_0 =0,& & x_{j+1}  := \int_0^t F_x(\nabla y_j(s), y_j(s), x_j(s), \overline{b}) \ \d s.
\end{eqnarray*}
%
%
We break down the proof in several steps. To begin with, we derive some estimates for the approximating sequences $(u_{j})_{j \in \mathbb{N}}$, $(y_{j})_{j \in \mathbb{N}}$, and $(x_{j})_{j \in \mathbb{N}}$. In the following, the constant $C>0$ will be generic and independent of time.

\subsubsection{Estimates}

For $0<T \leq \infty$ and $\omega > 0$ being the minimum $\omega$ appearing in Proposition~\ref{prop:smoothing}, let us define the quantities
\begin{align*}
k_j^u(T) &:= \sup_{0< s < T} e^{\frac{\omega s}{2}} s^{\tfrac{3}{2}\left(\tfrac{1}{p} - \tfrac{1}{q} \right)} \norm{u_j(s)}_{L^q_{\sigma}(\Omega)}, && k_j^{y}(T) := \sup_{0< s < T} \norm{ y_j(s)}_{L^{\infty}(\Omega)^3}, \\
\ k_j^{\nabla y}(T) &:= \sup_{0< s < T} e^{\frac{\omega s}{2}} s^{\tfrac{3}{2}\left(\tfrac{1}{p} - \tfrac{1}{q} \right)} \norm{\nabla y_j(s)}_{L^q(\Omega)^{3 \times 3}}, &&  k_j^{x}(T) := \sup_{0< s < T} \lvert x_j(s) \rvert.
\end{align*}
In the following, we will inductively show that all of these four quantities are finite and we will derive recursive inequalities relating these quantities at step $j + 1$ with the ones at step $j$ and zero. The finiteness for $j = 0$ is proven in the following lemma.

\begin{lemma}\label{lemma:initialvalues}
For all $3 \leq p \leq q < 3 + \eps$, there exists a constant $C > 0$ such that for all $0 < T \leq \infty$,
\begin{equation}
 k_0^u (T) + k_0^{\nabla y} (T) \leq C \Big(\norm{a}_{L^{p}_{\sigma}(\Omega)} +\norm{\nabla b}_{L^{p}(\Omega)^{3 \times 3}} \Big).
\label{conv-24}
\end{equation}
Moreover, $k_0^x (T) = 0$ and $k_0^y (T) \leq 2 \| b \|_{L^{\infty} (\Omega)^3}$.
\end{lemma}
\begin{proof}

We note that by Proposition~\ref{prop:smoothing}
\begin{equation}
\begin{aligned}
e^{\frac{\omega t}{2}} t^{\frac{3}{2}(\frac{1}{p}-\frac{1}{q})} \norm{ e^{-tA} a}_{L^{q}_{\sigma}(\Omega)}&\leq C e^{-\frac{\omega t}{2}} \norm{  a}_{L^{p}_{\sigma}(\Omega)} \leq C \norm{  a}_{L^{p}_{\sigma}(\Omega)},
\end{aligned}
\notag
\end{equation}
and hence 
\[k^{u}_{0}(T) \leq C \norm{ a}_{L^{p}_{\sigma}(\Omega)}. \]
Concerning the heat semigroup, we use the square root property of the Laplacian first, cf.\@ Proposition~\ref{Prop: Square root properties}, apply then Proposition~\ref{prop:smoothing} together with the fact that $\mathrm{Rg}(B^{\frac{1}{2}}) \subset L^p_0 (\Omega)^3$, and finally Proposition~\ref{Prop: Square root properties} again to deduce
\begin{equation}
\begin{aligned}
e^{\frac{\omega t}{2}} t^{\frac{3}{2}(\frac{1}{p}-\frac{1}{q})} \norm{\nabla e^{-tB} b_{s}}_{L^{q}(\Omega)^{3 \times 3}} &\leq C e^{\frac{\omega t}{2}} t^{\frac{3}{2}(\frac{1}{p}-\frac{1}{q})} \norm{e^{-tB} B^{\frac{1}{2}} b_{s}}_{L^{q}(\Omega)^3}
\leq  C \norm{\nabla b_{s}}_{L^{p}(\Omega)^{3 \times 3}}.
\end{aligned}
\notag
\end{equation}
Therefore, recalling that $\nabla b_{s}=\nabla b $, we have
\begin{align*}
 k^{\nabla y}_{0}(T) &\leq C \norm{\nabla b}_{L^{p}(\Omega)^{3 \times 3}}.
\end{align*}
\indent It is clear that $k_0^x (T)$ vanishes by its very definition. Moreover, for $k_0^y (T)$ we find by Proposition~\ref{prop:Linfty}
\begin{align*}
 \| e^{- t B} b_s \|_{L^{\infty} (\Omega)^3} \leq \| b_s \|_{L^{\infty} (\Omega)^3} &\leq 2 \| b \|_{L^{\infty} (\Omega)^3}. \qedhere
\end{align*}
\end{proof}

Notice that the choice of $W^{1 , p} (\Omega)^3 \cap L^{\infty} (\Omega)^3$ as the initial data space for $d$ is crucial for proving Lemma~\ref{lemma:initialvalues}. \par
For $3 \leq p < q < 3 +\eps$, we derive the following estimate for the sequence $(u_{j})_{j \in \mathbb{N}}$ using Proposition~\ref{prop:smoothing} (a) for $q$ and $\tfrac{q}{2}$, and H\"older's inequality for $\tfrac{2}{q}=\tfrac{1}{q}+\tfrac{1}{q}$
\begin{equation}
\begin{aligned}
\norm{u_{j+1}}_{L_{\sigma}^q(\Omega)}  &\leq  \norm{u_{0}}_{L_{\sigma}^q(\Omega)} + \Big\| \int_0^t  e^{-(t-s)A} \PP \div (u_j\otimes u_j) + e^{-(t-s)A} \PP \div([\nabla y_j]^{\top} \nabla y_j) \ \d s \Big\|_{L_{\sigma}^{q}(\Omega)}  \\
&\leq  \norm{u_{0}}_{L_{\sigma}^q(\Omega)} + C \int_0^t e^{-(t-s)\omega} (t-s)^{-\tfrac{1}{2}-\tfrac{3}{2q}} \big( \norm{u_j \otimes u_j}_{L_{\sigma}^{q/2}(\Omega)} + \norm{[\nabla y_j]^{\top} \nabla y_j}_{L^{q/2}(\Omega)^{3 \times 3}}  \big) \ \d s\\
&\leq  \norm{u_{0}}_{L_{\sigma}^q(\Omega)} + C \int_0^t e^{-t \omega} (t-s)^{-\tfrac{1}{2}-\tfrac{3}{2q}} s^{-3(\frac{1}{p}-\frac{1}{q})} \\
&\qquad \qquad \qquad \qquad \Big\{ \Big( e^{\frac{s\omega}{2}} s^{\frac{3}{2}(\frac{1}{p}-\frac{1}{q})} \norm{u_j}_{L_{\sigma}^{q}(\Omega)}\Big)^2 
  + \Big( e^{\frac{s \omega}{2}} s^{\frac{3}{2}(\frac{1}{p}-\frac{1}{q})}\norm{\nabla y_j}_{L^{q}(\Omega)^{3 \times 3}}  \Big)^2 \Big\} \ \d s \\
&\leq \norm{u_{0}}_{L_{\sigma}^q(\Omega)}+ C \Big(e^{-t \omega} \int_{0}^{t} (t-s)^{-\frac{1}{2}-\frac{3}{2q}} s^{-3(\frac{1}{p}-\frac{1}{q})} \ \d s \Big) [k_{j}^{u}(T)^2+k_{j}^{\nabla y}(T)^2]
\end{aligned}
\notag
\end{equation}
which implies, multiplying by the factor $e^{\frac{\omega t}{2}} t^{\frac{3}{2}(\frac{1}{p}-\frac{1}{q})}$ and taking $\sup_{0< t < T}$, that
\begin{equation}
k^{u}_{j+1}(T) \leq k^{u}_{0}(T)+ C \Big(\sup_{0< t< T} e^{-\frac{\omega t}{2}} t^{\frac{3}{2}(\frac{1}{p}-\frac{1}{q})}\int_{0}^{t} (t-s)^{-\frac{1}{2}-\frac{3}{2q}} s^{-3(\frac{1}{p}-\frac{1}{q})} \ \d s \Big) [k_{j}^{u}(T)^2+k_{j}^{\nabla y}(T)^2].
\label{conv-1}
\end{equation}
Since $3\leq p<q<3+\eps$, it follows that 
$\frac{1}{2}-\frac{3}{2p} \geq 0,\ 1-3(\frac{1}{p}-\frac{1}{q}) >0,\ \frac{1}{2}-\frac{3}{2q}>0$, 
and hence 
\[
\begin{aligned}
\sup_{0< t< T} e^{-\frac{\omega t}{2}} t^{\frac{3}{2}(\frac{1}{p}-\frac{1}{q})}\int_{0}^{t} (t-s)^{-\frac{1}{2}-\frac{3}{2q}} s^{-3(\frac{1}{p}-\frac{1}{q})} \ \d s &=\left(\sup_{0< t< T} e^{-\frac{\omega t}{2}} t^{\frac{1}{2}-\frac{3}{2p}}\right) B \big(1-3\big(\tfrac{1}{p}-\tfrac{1}{q}\big),\tfrac{1}{2}-\tfrac{3}{2q}\big),
 \end{aligned}
 \]
where $B(x,y)$ denotes the beta function for $x,y>0 $. 
Therefore, setting $C_{1}(T):= \sup_{0<t<T} e^{-\frac{\omega t}{2}} t^{\frac{1}{2}-\frac{3}{2p}}$, equation~\eqref{conv-1} turns into
\begin{equation}
\begin{aligned}
k^{u}_{j+1}(T) &\leq k^{u}_{0}(T)+ C C_{1}(T) B \big(1-3\big(\tfrac{1}{p}-\tfrac{1}{q}\big),\tfrac{1}{2}-\tfrac{3}{2q}\big) [k_{j}^{u}(T)^2+k_{j}^{\nabla y}(T)^2].
\end{aligned}
\label{conv-2}
\end{equation}
Similarly, for $\nabla y_{j + 1}$, but now using Proposition~\ref{prop:smoothing} (b) and H\"older's inequality we obtain 
\begin{equation}
\begin{aligned}
\norm{\nabla y_{j+1}}_{L^q(\Omega)^{3 \times 3}}  &\leq  \norm{\nabla y_{0}}_{L^q(\Omega)^{3 \times 3}} + \Big\| \int_0^t  \nabla e^{-(t-s)B} \big( (u_j \cdot \nabla ) y_j -  P_s \abs{\nabla y_j}^2 (x_j+y_j+\overline{b}) \big) \ \d s \Big\|_{L^{q}(\Omega)^{3 \times 3}}  \\
&\leq  \norm{\nabla y_{0}}_{L^q(\Omega)^{3 \times 3}} + C \int_0^t e^{-(t-s)\omega} (t-s)^{-\tfrac{1}{2}-\tfrac{3}{2q}} \\
&\qquad \big(\norm{( u_j \cdot \nabla ) y_j }_{L^{q/2}(\Omega)^3} + \norm{\abs{\nabla y_j}^2}_{L^{q/2}(\Omega)} (\lvert x_j \rvert +\norm{y_j}_{L^{\infty}(\Omega)^3}+\rvert \overline{b} \rvert)  \big) \ \d s\\
&\leq  \norm{\nabla y_{0}}_{L^q(\Omega)^{3 \times 3}} + C \int_0^t e^{-t\omega} (t-s)^{-\frac{1}{2}-\frac{3}{2q}}  s^{-3(\frac{1}{p}-\frac{1}{q})} \ \d s  \\
&\qquad \Big\{ \Big(\sup_{0 < s < T} e^{\frac{s \omega}{2}} s^{\frac{3}{2}({\frac{1}{p}-\frac{1}{q}})} \norm{u_j}_{L_{\sigma}^{q}(\Omega)}\Big)\Big(\sup_{0 < s < T} e^{\frac{s \omega}{2}}  s^{\frac{3}{2}(\frac{1}{p}-\frac{1}{q})}\norm{\nabla y_j}_{L^{q}(\Omega)^{3 \times 3}}  \Big)\\ 
 & \qquad +\Big(\sup_{0 < s < T} e^{\frac{s\omega}{2}} s^{\frac{3}{2}(\frac{1}{p}-\frac{1}{q})}\norm{\nabla y_j}_{L^{q}(\Omega)^{3 \times 3}}  \Big)^2  (\sup_{0 < s < T} \lvert x_j(s) \rvert +\sup_{0 < s < T}\norm{y_j}_{L^{\infty}(\Omega)^3}+ \lvert \overline{b} \rvert) \Big\}.
\end{aligned}
\label{conv-3}
\end{equation}
Proceeding as in the previous case, we obtain
\begin{align}
\begin{aligned}
k^{\nabla y}_{j+1}(T)&\leq k^{\nabla y}_{0}(T) + C C_{1}(T) B \big(1-3\big(\tfrac{1}{p}-\tfrac{1}{q}\big),\tfrac{1}{2}-\tfrac{3}{2p}\big) \\
& \qquad \left[k^{u}_{j}(T) k^{\nabla y}_{j}(T) + k^{\nabla y}_{j}(T)^2 (k^{x}_{j}(T)+k^{y}_{j}(T)+\lvert \overline{b} \rvert) \right].
\end{aligned}
\label{Eq: Recursive inequality for gradient y}
\end{align}
Now, let us also consider analogous estimates for $r$ where $3 \leq \max \{p,\frac{q}{2}\} \leq r \leq q < 3+\eps$. Taking $q$ close enough to $p$, we can assure that $\frac{q}{2} \leq p $ and thus, that the choice $r = p$ is possible. \par
For the sequence $(u_{j})_{j \in \mathbb{N}} $, we obtain similar to the above by Proposition~\ref{prop:smoothing} (a)
\begin{equation}
\begin{aligned}
\norm{u_{j+1}}_{L_{\sigma}^r(\Omega)} 
&\leq  \norm{u_{0}}_{L_{\sigma}^r(\Omega)} + C \int_0^t e^{-t \omega} (t-s)^{-\frac{1}{2}-\frac{3}{2}(\frac{2}{q}-\frac{1}{r})} s^{-3(\frac{1}{p}-\frac{1}{q})} \\
&\qquad \qquad \qquad \qquad \Big\{ \left( e^{\frac{s\omega}{2}} s^{\frac{3}{2}(\frac{1}{p}-\frac{1}{q})} \norm{u_j}_{L_{\sigma}^{q}(\Omega)}\right)^2 
  + \left( e^{\frac{s \omega}{2}} s^{\frac{3}{2}(\frac{1}{p}-\frac{1}{q})}\norm{\nabla y_j}_{L^{q}(\Omega)^{3 \times 3}}  \right)^2 \Big\} \ \d s, \\
\end{aligned}
\notag
\end{equation}
and this implies
\begin{equation}
\begin{aligned}
\sup_{0<t<T} e^{\frac{\omega t}{2}} t^{\frac{3}{2}(\frac{1}{p}-\frac{1}{r})} &\norm{u_{j+1}}_{L_{\sigma}^r(\Omega)} \leq \sup_{0<t<T} e^{\frac{\omega t}{2}} t^{\frac{3}{2}(\frac{1}{p}-\frac{1}{r})} \norm{u_0}_{L_{\sigma}^r(\Omega)} \\
&+ C \left(\sup_{0<t<T} e^{-\frac{\omega t}{2}} t^{\frac{1}{2}-\frac{3}{2p}} \right) B \big(1-3\big(\tfrac{1}{p}-\tfrac{1}{q}\big),\tfrac{1}{2}-\tfrac{3}{2}\big(\tfrac{2}{q}-\tfrac{1}{r} \big)\big)  [k_{j}^{u}(T)^2+k_{j}^{\nabla y}(T)^2].
\end{aligned}
\label{conv-100}
\end{equation}
Similarly, it follows for $\nabla y_{j + 1}$, that
\begin{equation}
\begin{aligned}
&\sup_{0<t<T} e^{\frac{\omega t}{2}} t^{\frac{3}{2}(\frac{1}{p}-\frac{1}{r})} \norm{\nabla y_{j+1}}_{L^r(\Omega)^{3 \times 3}} \leq \sup_{0<t<T} e^{\frac{\omega t}{2}} t^{\frac{3}{2}(\frac{1}{p}-\frac{1}{r})} \norm{\nabla y_0}_{L^r(\Omega)^{3 \times 3}} \\
&+ C \Big(\sup_{0<t<T} e^{- \frac{\omega t}{2}} t^{\frac{1}{2}-\frac{3}{2p}} \Big) B \big(1-3\big(\tfrac{1}{p}-\tfrac{1}{q}\big),\tfrac{1}{2}-\tfrac{3}{2}\big(\tfrac{2}{q}-\tfrac{1}{r} \big) \big)  \Big[k^{u}_{j}(T) k^{\nabla y}_{j}(T) + k^{\nabla y}_{j}(T)^2 (k^{x}_{j}(T)+k^{y}_{j}(T)+\lvert \overline{b} \rvert) \Big].
\end{aligned}
\label{conv-101}
\end{equation}
Next we estimate $\norm{y_{j + 1}}_{L^{\infty}(\Omega)^3}$, by virtue of Propositions~\ref{prop:Linfty} and~\ref{prop:smoothing} (b) and H\"older's inequality, as
\begin{equation}
\begin{aligned}
\norm{y_{j+1}}_{L^{\infty}(\Omega)^3}  &\leq  \|y_{0}\|_{L^{\infty}(\Omega)^3} + \Big\| \int_0^t  e^{-(t-s)B} \big( (u_j \cdot \nabla) y_j -  P_s \abs{\nabla y_j}^2 (x_j+y_j+\overline{b})\big) \; \d s \Big\|_{L^{\infty}(\Omega)^3} \\
&\leq  \norm{y_{0}}_{L^{\infty}(\Omega)^3} + C \int_0^t e^{-(t-s)\omega} (t-s)^{-\frac{3}{q}} \\
&\qquad \qquad  \big(\norm{u_j}_{L_{\sigma}^{q}(\Omega)}\norm{\nabla y_j}_{L^{q}(\Omega)^{3 \times 3}} + \norm{\nabla y_{j}}^2_{L^{q}(\Omega)^{3 \times 3}} (\lvert x_j \rvert +\norm{y_j}_{L^{\infty}(\Omega)^3}+ \lvert \overline{b} \rvert)  \big) \ \d s \\
&\leq  \norm{y_{0}}_{L^{\infty}(\Omega)^3} + C \int_0^t  e^{-t\omega}(t-s)^{-\frac{3}{q}}  s^{-3(\frac{1}{p}-\frac{1}{q})} \ \d s  \\
&\qquad \Big\{ \Big(\sup_{0 < s < T} e^{\frac{s\omega}{2}} s^{\frac{3}{2}(\frac{1}{p}-\frac{1}{q})} \norm{u_j}_{L_{\sigma}^{q}(\Omega)}\Big)\Big(\sup_{0 < s < T} e^{\frac{s \omega}{2}} s^{\frac{3}{2}(\frac{1}{p}-\frac{1}{q})}\norm{\nabla y_{j}}_{L^{q}(\Omega)^{3 \times 3}} \Big)\\ 
 &\qquad + \Big(\sup_{0 < s < T}  e^{\frac{s\omega}{2}} s^{\frac{3}{2}(\frac{1}{p}-\frac{1}{q})}\norm{\nabla y_j}_{L^{q}(\Omega)^{3 \times 3}}  \Big)^2 
  \big( \sup_{0 < s < T} \lvert x_j(s) \rvert + \sup_{0 < s < T}\norm{y_j}_{L^{\infty}(\Omega)^3}+ \lvert \overline{b} \rvert \big) \Big\},
\end{aligned}
\label{conv-5}
\end{equation}
and hence, setting $C_{2}(T):=\sup_{0<t<T} e^{-\omega t} t^{1-\frac{3}{p}}$, we have
\begin{equation}
\begin{aligned}
k^{y}_{j+1}(T)&\leq k^{y}_{0}(T)+ C C_{2}(T) B \big(1-3 \big(\tfrac{1}{p}-\tfrac{1}{q} \big),1-\tfrac{3}{q} \big) \Big[k^{u}_{j}(T) k^{\nabla y}_{j}(T) + k^{\nabla y}_{j}(T)^2 (k^{x}_{j}(T)+k^{y}_{j}(T)+ \lvert \overline{b} \rvert ) \Big].
\end{aligned}
\label{conv-6}
\end{equation}
Finally for $\lvert x_{j + 1} \rvert$ one obtains using H\"older's inequality and the embedding $L^q(\Omega)\hookrightarrow L^2(\Omega)$
\begin{equation}
\begin{aligned}
 \lvert x_{j+1} \rvert &\leq \int_{0}^{t} \Big\vert \frac{1}{\vert \Omega \vert} \int_{\Omega} \vert \nabla y_{j} \vert^2 (x_{j}+y_{j}+\overline{b})  \Big\vert \ \d s 
 \leq \frac{1}{\vert \Omega \vert} \int_0^t \norm{\nabla y_j}_{L^2(\Omega)^{3 \times 3}}^2 ( \lvert x_{j} \rvert +\norm{y_{j}}_{L^{\infty}(\Omega)^3}+ \lvert \overline{b} \rvert ) \ \d s\\
 &\leq C \int_0^t \norm{\nabla y_j}_{L^q(\Omega)^{3 \times 3}}^2 ( \lvert x_{j} \rvert +\norm{y_{j}}_{L^{\infty}(\Omega)^3}+ \lvert \overline{b} \rvert ) \; \d s\\
&\leq C \int_0^t  e^{-s \omega} s^{-3(\frac{1}{p}-\frac{1}{q})} \; \d s \left(\sup_{0 < s < T} e^{\frac{s \omega}{2}} s^{\frac{3}{2}(\frac{1}{p}-\frac{1}{q})}\norm{\nabla y_j}_{L^{q}(\Omega)^{3 \times 3}} \right)^2 
\big( \sup_{0 < s < T} \lvert x_j \rvert +\sup_{0 < s < T} \norm{y_j}_{L^{\infty}(\Omega)^3} + \lvert \overline{b} \rvert \big)
\end{aligned}
\label{conv-7}
\end{equation}
and therefore with $C_{3}(T):=\sup_{0<t<T} \int_0^t e^{-s \omega} s^{-3(\frac{1}{p}-\frac{1}{q})} \; \d s $
\begin{equation}
\begin{aligned}
&k^{x}_{j+1}(T) \leq C C_{3}(T) k^{\nabla y}_{j}(T)^2 (k^{x}_{j}(T)+ k^{y}_{j}(T)+ \lvert \overline{b} \rvert ).
\end{aligned}
\label{conv-8}
\end{equation}

For $0<T\leq \infty$, set $\tilde{C}_T :=\max \{C_{1}(T), C_{2}(T), C_{3}(T) \}$.

\begin{remark}\label{rem:CT}
Note that $C_i(T)$, $i=1,2,3$, are continuous as functions in $T$, uniformly bounded on $(0,\infty)$, and monotonically increasing. In particular, 
\begin{itemize}
\item[(a)] for $T=+\infty$, $C_i(T)$ are well-defined constants, and
\item[(b)] for $ p>3$, $\lim_{T \to 0} C_i(T)=0$.
\end{itemize}
Notice that $\tilde{C}_T$ inherits these properties.
\end{remark}

Summarising the estimates~\eqref{conv-2},~\eqref{Eq: Recursive inequality for gradient y},~\eqref{conv-6}, and~\eqref{conv-8}, we arrive at the inequalities
\begin{equation}
\begin{aligned}
&k^{u}_{j+1}(T) \leq k^{u}_{0}(T)+ C \tilde{C}_T [k_{j}^{u}(T)^2+k_{j}^{\nabla y}(T)^2], \\
&k^{\nabla y}_{j+1}(T) \leq k^{\nabla y}_{0}(T)+ C \tilde{C}_T \Big[k^{u}_{j}(T) k^{\nabla y}_{j}(T) + k^{\nabla y}_{j}(T)^2 (k^{x}_{j}(T)+k^{y}_{j}(T)+ \lvert \overline{b} \rvert ) \Big], \\
&k^{y}_{j+1}(T) \leq k^{y}_{0}(T)+ C \tilde{C}_T \Big[k^{u}_{j}(T) k^{\nabla y}_{j}(T) + k^{\nabla y}_{j}(T)^2 (k^{x}_{j}(T)+k^{y}_{j}(T)+ \lvert \overline{b} \rvert ) \Big], \\
&k^{x}_{j+1}(T) \leq C \tilde{C}_T k^{\nabla y}_{j}(T)^2 (k^{x}_{j}(T)+ k^{y}_{j}(T)+ \lvert \overline{b} \rvert ).
\end{aligned}
\label{conv-9}
\end{equation}
By virtue of Lemma~\ref{lemma:initialvalues} this proves by induction that $k^{u}_{j+1}(T)$, $k^{\nabla y}_{j+1}(T)$, $k^{y}_{j+1}(T)$, and $k^{x}_{j+1}(T)$ are finite. Let us introduce the notations
\begin{equation}
k^{q}_{j}:= k^{u}_{j} + k^{\nabla y}_{j}, \quad k^{\infty}_{j}:= k^{y}_{j} + k^{x}_{j}.
\notag
\end{equation}

\subsubsection{Estimates on the differences}
Next, let us define 
\begin{align*}
\begin{aligned}
 &W_{j}(t) := u_{j+1}(t)-u_{j}(t), &&Z_{j}(t) :=\nabla y_{j+1}(t)-\nabla y_{j}(t),  \\
 &Y_{j}(t) :=y_{j+1}(t)-y_{j}(t) , &&X_{j}(t) :=x_{j+1}(t)-x_{j}(t),
\end{aligned}
\notag
\end{align*}
and the corresponding time-weighted quantities
\begin{align*}
\delta_j^u(T) &:= \sup_{0< s < T} e^{\frac{\omega s}{2}} s^{\tfrac{3}{2}\left(\tfrac{1}{p} - \tfrac{1}{q} \right)} \norm{W_j(s)}_{L^q_{\sigma}(\Omega)}, \quad &&\delta_j^{\nabla y}(T) := \sup_{0< s < T} e^{\frac{\omega s}{2}} s^{\tfrac{3}{2}\left(\tfrac{1}{p} - \tfrac{1}{q} \right)} \norm{Z_j(s)}_{L^q(\Omega)^{3 \times 3}}, \\
\delta_j^{y}(T) &:= \sup_{0< s < T} \norm{Y_j(s)}_{L^{\infty}(\Omega)^3}, \quad && \delta_j^{x}(T) := \sup_{0< s < T} \lvert X_j(s) \rvert.
\end{align*}
Now, using the bi-linearity of the tensor product
\begin{equation}
\begin{aligned}
W_{j}(t)
= - \int_{0}^{t} e^{-(t-s)A} \PP \div\Big[  (W_{j-1} \otimes u_{j} +u_{j-1} \otimes W_{j-1} ) +  (Z_{j-1}^{\top} \nabla y_{j}+ [\nabla y_{j-1}]^{\top} Z_{j-1}) \Big] \ \d s
\end{aligned}
\notag
\end{equation}
and therefore 
proceeding as in the derivation of~\eqref{conv-1} and~\eqref{conv-9}, we arrive at
\begin{equation}
\begin{aligned}
\delta^{u}_{j}(T) \leq C \tilde{C}_{T} \Big[\delta^{u}_{j-1}(T)(k^{q}_{j}(T)+k^{q}_{j-1}(T)) + \delta^{\nabla y}_{j-1}(T)(k^{q}_{j}(T)+k^{q}_{j-1}(T)) \Big].
\end{aligned}
\label{conv-10}
\end{equation}
To estimate $Z_{j}$, we write
\begin{align*}
 Z_j &= \int_{0}^{t} \nabla e^{-(t-s)B} [W_{j-1} \cdot \nabla y_{j}+u_{j-1} \cdot Z_{j-1}] -P_{s} \big[ \big\{ (\lvert \nabla y_j \rvert - \lvert \nabla y_{j - 1} \rvert) \vert \nabla y_{j} \vert \\
 &\qquad+ \vert \nabla y_{j-1} \vert (\lvert \nabla y_j \rvert - \lvert \nabla y_{j - 1} \rvert) \big\}(x_{j}+y_{j}+\overline{b}) + \vert \nabla y_{j-1} \vert^2 \{(x_{j} - x_{j - 1} +y_{j} - y_{j - 1}) \} \big] \; \d s
\end{align*}
and estimate analogously to~\eqref{conv-3} and~\eqref{conv-9} using in addition $\big\lvert \lvert \nabla y_j \rvert - \lvert \nabla y_{j - 1} \rvert \big\rvert \leq \lvert Z_{j - 1} \rvert$. This yields the inequality
\begin{equation}
\begin{aligned}
\delta^{\nabla y}_{j}(T) &\leq C \tilde{C}_{T} \Big[\delta^{u}_{j-1}(T) k^{q}_{j}(T) +k^{q}_{j-1}(T) \delta^{\nabla y}_{j-1}(T) \\
&\quad + \big( \delta^{\nabla y}_{j-1}(T) k^{q}_{j}(T)+ k^{q}_{j-1}(T) \delta^{\nabla y}_{j-1}(T) \big) (2 k^{\infty}_{j}(T)+ \lvert \overline{b} \lvert)+k^{q}_{j-1}(T)^2 \big(\delta^{x}_{j-1}(T)+\delta^{y}_{j-1}(T) \big) \Big].
\end{aligned}
\label{conv-12}
\end{equation} 
The term $Y_j$ can be written similarly to $Z_j$ but without the gradient in front of the semigroup. Thus, following~\eqref{conv-5} and the derivation of~\eqref{conv-9} it follows
\begin{equation}
\begin{aligned}
\delta^{y}_{j}(T) &\leq C \tilde{C}_{T} \Big[\delta^{u}_{j-1}(T) k^{q}_{j}(T) +k^{q}_{j-1}(T) \delta^{\nabla y}_{j-1}(T) \\
&\qquad  + \big( \delta^{\nabla y}_{j-1}(T) k^{q}_{j}(T)+ k^{q}_{j-1}(T) \delta^{\nabla y}_{j-1}(T) \big) (2 k^{\infty}_{j}(T)+ \lvert \overline{b} \rvert)+k^{q}_{j-1}(T)^2 \big(\delta^{x}_{j-1}(T)+\delta^{y}_{j-1}(T) \big) \Big].
\end{aligned}
\label{conv-13}
\end{equation}
In order to deal with the term $X_{j}$, we observe that
\begin{equation}
\begin{aligned}
X_{j}(t) 
&= \int_{0}^{t}  \frac{1}{\vert \Omega \vert} \int_{\Omega} \big\{ (\lvert \nabla y_j \rvert - \lvert \nabla y_{j - 1} \rvert) \vert \nabla y_{j} \vert+\vert \nabla y_{j-1} \vert (\lvert \nabla y_j \rvert - \lvert \nabla y_{j - 1} \rvert) \big\} (x_{j}+y_{j}+\overline{b}) \\
 &\qquad \qquad + \vert \nabla y_{j-1} \vert^{2}(x_j - x_{j-1} + y_j - y_{j-1}) \; \d x \; \d s
\end{aligned}
\notag
\end{equation}
and hence
by proceeding as in~\eqref{conv-7} and the derivation of~\eqref{conv-9}, we deduce
\begin{equation}
\begin{aligned}
\delta^{x}_{j}(T) &\leq C \tilde{C}_{T} \Big[ \big(\delta^{\nabla y}_{j-1}(T) k^{q}_{j}(T)+ k^{q}_{j-1}(T) \delta^{\nabla y}_{j-1}(T) \big) (2 k^{\infty}_{j}(T) + \lvert \overline{b} \rvert)\\
&\qquad \qquad \qquad \qquad \qquad \qquad \qquad +k^{q}_{j-1}(T)^2 \big(\delta^{x}_{j-1}(T)+\delta^{y}_{j-1}(T) \big) \Big].
\end{aligned}
\label{conv-14}
\end{equation}
Using the first two inequalities in~\eqref{conv-9}, an estimate for $k_{j + 1}^q$ in terms of $k_0^q$, $k_j^q$, and $k_j^{\infty}$ can be deduced by Young's inequality and by absorbing the factors into the constant $C>0$ as follows 
\begin{equation}
\begin{aligned}
k^{q}_{j+1}(T)&\leq k^{q}_{0}(T)+ C \tilde{C}_{T} \Big[k^{u}_{j}(T)^2+ k^{u}_{j}(T) k^{\nabla y}_{j}(T)+k^{\nabla y}_{j}(T)^2 (1+k^{\infty}_{j}(T)+ \lvert \overline{b} \rvert) \Big] \\
&\leq k^{q}_{0}(T)+ C \tilde{C}_{T} \Big[k^{q}_{j}(T)^2+k^{q}_{j}(T)^2(1+k^{\infty}_{j}(T) + \lvert \overline{b} \rvert) \Big]\\
&\leq k^{q}_{0}(T)+ C \tilde{C}_{T} \Big[k^{q}_{j}(T)^2(1+k^{\infty}_{j}(T) + \lvert \overline{b} \rvert) \Big].
\end{aligned}
\label{conv-15}
\end{equation}
Similarly, by virtue of the last two inequalities in~\eqref{conv-9}, we obtain
\begin{equation}
\begin{aligned}
k^{\infty}_{j+1}(T) &\leq k^{\infty}_{0}(T)+C \tilde{C}_{T} \Big[k^{u}_{j}(T) k^{\nabla y}_{j}(T)+ k^{\nabla y}_{j}(T)^2(k^{\infty}_{j}(T)+ \lvert \overline{b} \rvert ) \Big] \\
&\leq k^{\infty}_{0}(T)+C \tilde{C}_{T} \Big[k^{q}_{j}(T)^2 (1+k^{\infty}_{j}(T)+ \lvert \overline{b} \rvert ) \Big].
\end{aligned}
\label{conv-16}
\end{equation}
Let us denote $\delta_{j}(T):= \delta^{u}_{j}(T)+\delta^{\nabla y}_{j}(T)+\delta^{y}_{j}(T)+\delta^{x}_{j}(T)$. Now, suppose that there are constants $K_{1}, K_{2}>0 $ (possibly depending on $T$), without loss of generality let $K_{2}\geq 1$, satisfying
\begin{align}
\label{Eq: Boundedness of sequences}
 k^{q}_{0}(T)\leq \frac{K_{1}}{2},\quad k^{q}_{j}(T) < K_{1},\quad k^{\infty}_{0}(T) \leq \frac{K_{2}}{2},\quad k^{\infty}_{j}(T) < K_{2},\quad  \lvert \overline{b} \rvert \leq K_{2}.
\end{align}
Then,~\eqref{conv-10} implies
\begin{equation}
\begin{aligned}
\delta^{u}_{j}(T)&\leq C \tilde{C}_{T} (k^{q}_{j}(T)+k^{q}_{j-1}(T)) \delta_{j-1}(T) < 2 K_{1} C \tilde{C}_{T} \delta_{j-1}(T).
\end{aligned}
\label{conv-17}
\end{equation}
From~\eqref{conv-12}, we have
\begin{equation}
\begin{aligned}
\delta^{\nabla y}_{j}(T)&\leq C \tilde{C}_{T} \Big[(k^{q}_{j}(T)+k^{q}_{j-1}(T))+(k^{q}_{j}(T)+k^{q}_{j-1}(T))(2 k^{\infty}_{j}(T) + \lvert \overline{b} \rvert)+k^{q}_{j-1}(T)^2 \Big] \delta_{j-1}(T) \\
&<  C \tilde{C}_{T} [2 K_{1}+6 K_{1} K_{2}+ K_{1}^2] \delta_{j-1}(T).
\end{aligned}
\label{conv-18}
\end{equation}
Similarly,~\eqref{conv-13} implies
 \begin{equation}
\begin{aligned}
\delta^{ y}_{j}(T)
&< C \tilde{C}_{T} [2 K_{1}+6 K_{1} K_{2}+ K_{1}^2] \delta_{j-1}(T)
\end{aligned}
\label{conv-19}
\end{equation}
and~\eqref{conv-14} implies
\begin{equation}
\begin{aligned}
\delta^{x}_{j}(T)
&< C \tilde{C}_{T} [6 K_{1} K_{2}+ K_{1}^2] \delta_{j-1}(T).
\end{aligned}
\label{conv-20}
\end{equation}
Combining~\eqref{conv-17}-\eqref{conv-20}, we obtain
\begin{equation}
\begin{aligned}
\delta_{j}(T) < 6 C \tilde{C}_{T} [2K_{1}+6K_{1}K_{2}+K_{1}^2] \delta_{j-1}(T).
\end{aligned}
\label{conv-21}
\end{equation}
Therefore, if we can show that $6 C \tilde{C}_{T} [2K_{1}+6K_{1}K_{2}+K_{1}^2]<1 $, then~\eqref{conv-21} gives us a contraction, i.e., for some $\theta \in (0 , 1)$
\begin{align}
\label{Eq: Contraction of delta}
 \delta_j < \theta \delta_{j - 1} \qquad (j \in \N).
\end{align}
Next, note that if $6C \tilde{C}_{T} K_{1} K_{2} <1$, then~\eqref{conv-15} together with~\eqref{Eq: Boundedness of sequences} implies
\begin{equation}
k^{q}_{j+1}(T) \leq k^{q}_{0}(T)+3 C \tilde{C}_{T} K_{1}^2 K_{2} < \frac{K_{1}}{2}+ \frac{K_{1}}{2} = K_{1}. 
\notag
\end{equation}
Similarly, if $6C \tilde{C}_{T} K_{1}^2  <1 $, then~\eqref{conv-16} together with~\eqref{Eq: Boundedness of sequences} implies
\begin{equation}
k^{\infty}_{j+1}(T) \leq k^{\infty}_{0}(T)+3 C \tilde{C}_{T} K_{1}^2 K_{2} < \frac{K_{2}}{2}+ \frac{K_{2}}{2} = K_{2}. 
\notag
\end{equation}
We further note that both of the conditions above are also fulfilled if $6 C \tilde{C}_{T} [2K_{1}+6K_{1}K_{2}+K_{1}^2]<1 $. Let us define 
\begin{align}
\label{Eq: Definition of bounds}
 K_{1}:= 2 k^{q}_{0}(T) \quad \text{and} \quad K_{2}:= \max\{2 k^{\infty}_{0}(T), \lvert \overline{b} \rvert\}. 
\end{align}
We already know by Lemma~\ref{lemma:initialvalues} that $k_0^{\infty} \leq 2 \| b \|_{L^{\infty} (\Omega)^3}$ so that $K_{2} \leq 4 \| b \|_{L^{\infty} (\Omega)^3}$. Now,
\begin{equation}
\begin{aligned}
6C \tilde{C}_{T} \Big[2K_{1}+6K_{1}K_{2}+K_{1}^2 \Big] &\leq 6C \tilde{C}_{T} \Big[3 K + 24 K \| b \|_{L^{\infty} (\Omega)^3} \Big] 
= 144 C \tilde{C}_{T} K (1 + \| b \|_{L^{\infty} (\Omega)^3}), 
\end{aligned}
\notag
\end{equation}
where $K:=\max \{K_{1},K_{1}^2 \} $.
Therefore, if 
\begin{equation}
144 C \tilde{C}_{T} K (1 + \| b \|_{L^{\infty} (\Omega)^3}) < 1,
\label{conv-22}
\end{equation} 
then all the conditions are fulfilled and we have a contraction in terms of $\delta_{j}$ in the sense of~\eqref{Eq: Contraction of delta}.
%
\subsubsection{Conditions on the initial data and global existence}

In the following, we show that~\eqref{conv-22} is valid under the present hypotheses of the theorem. The following lemma is crucial for the validity of~\eqref{conv-22} for small times $T$.

\begin{lemma}\label{small}
If $q > p$, the term $k^{q}_{0}(T) $ admits the following behaviour 
\begin{equation}
k^{q}_{0}(T) \rightarrow 0 \quad \text{as}\quad T \rightarrow 0. 
\label{conv-23}
\end{equation}
\end{lemma}

In the following lemma, we recall that a Sobolev function with average zero can be approximated by a sequence of smooth functions with average zero. This short proof is left for the reader.

\begin{lemma}\label{approx}
Let $b \in W^{1 , p} (\Omega)^3$ be a function with average zero. Then there exists a sequence $(b_{n})_{n \in \N} \subset C^{\infty} (\overline{\Omega})^3 $ such that 
\[b_{n} \to b \quad \text{in} \quad W^{1 , p} (\Omega)^3 \]
as $n \to \infty$ and $b_{n} $ has average zero for every $n \in \N$.
\end{lemma}

\begin{proof}[Proof of Lemma~\ref{small}]
Let $(a^j)_{j \in \mathbb{N}}$ be a sequence in $C_{c , \sigma}^{\infty}(\Omega)$ such that $a^j \rightarrow a$ in $L^p_{\sigma} (\Omega)$. Then,
\[ e^{\frac{\omega t}{2}} t^{\frac{3}{2}(\frac{1}{p}-\frac{1}{q})} \norm{e^{-tA} a}_{L^{q}_{\sigma}(\Omega)} \leq e^{\frac{\omega t}{2}} t^{\frac{3}{2}(\frac{1}{p}-\frac{1}{q})} \norm{e^{-tA} (a - a^{j})}_{L^{q}_{\sigma}(\Omega)}+ e^{\frac{\omega t}{2}} t^{\frac{3}{2}(\frac{1}{p}-\frac{1}{q})} \norm{e^{-tA} a^{j}}_{L^{q}_{\sigma}(\Omega)}.\]
Estimate by virtue of Proposition~\ref{prop:smoothing}
\begin{equation}
\begin{aligned}
e^{\frac{\omega t}{2}} t^{\frac{3}{2}(\frac{1}{p}-\frac{1}{q})} \norm{e^{-tA} (a - a^{j})}_{L^{q}_{\sigma}(\Omega)} &\leq C e^{-\frac{\omega t}{2}} t^{\frac{3}{2}(\frac{1}{p}-\frac{1}{q})} \cdot t^{-\frac{3}{2}(\frac{1}{p}-\frac{1}{q})} \norm{a - a^{j}}_{L^{p}_{\sigma}(\Omega)} \leq C \norm{a - a^{j}}_{L^{p}_{\sigma}(\Omega)},
\end{aligned}
\notag
\end{equation}
and use that the Stokes semigroup is exponentially stable on $L^q_{\sigma} (\Omega)$ to deduce
\begin{equation}
\begin{aligned}
e^{\frac{\omega t}{2}} t^{\frac{3}{2}(\frac{1}{p}-\frac{1}{q})} \norm{e^{-tA} a^{j}}_{L^{q}_{\sigma}(\Omega)}&\leq C e^{-\frac{\omega t}{2}} t^{\frac{3}{2}(\frac{1}{p}-\frac{1}{q})} \norm{a^{j}}_{L^q_{\sigma} (\Omega)}.
\end{aligned}
\notag
\end{equation}
%
Similarly, choosing an approximating sequence of smooth functions $(b_{s}^{j})_{j \in \mathbb{N}}$ to $b_{s}$ (cf.\@ Lemma~\ref{approx}), we can write 
\[ e^{\frac{\omega t}{2}} t^{\frac{3}{2}(\frac{1}{p}-\frac{1}{q})} \norm{\nabla e^{-tB} b_{s}}_{L^{q}(\Omega)^{3 \times 3}} \leq e^{\frac{\omega t}{2}} t^{\frac{3}{2}(\frac{1}{p}-\frac{1}{q})} \norm{\nabla e^{-tB} (b_{s}-b_{s}^{j})}_{L^{q}(\Omega)^{3 \times 3}}+ e^{\frac{\omega t}{2}} t^{\frac{3}{2}(\frac{1}{p}-\frac{1}{q})} \norm{\nabla e^{-tB} b_{s}^{j}}_{L^{q}(\Omega)^{3 \times 3}}. \]
Now, by Proposition~\ref{prop:smoothing} (b) 
\[\begin{aligned}
e^{\frac{\omega t}{2}} t^{\frac{3}{2}(\frac{1}{p}-\frac{1}{q})} \norm{\nabla e^{-tB} b_{s}^{j}}_{L^{q}(\Omega)^{3 \times 3}} 
\leq C t^{\frac{3}{2}(\frac{1}{p}-\frac{1}{q})} \norm{\nabla b_{s}^{j}}_{L^{q}(\Omega)^{3 \times 3}},
\end{aligned} \]
and also
\[\begin{aligned}
e^{\frac{\omega t}{2}} t^{\frac{3}{2}(\frac{1}{p}-\frac{1}{q})} \norm{\nabla e^{-tB} (b_{s}-b_{s}^{j})}_{L^{q}(\Omega)^{3 \times 3}} 
\leq C  \norm{\nabla (b_{s}-b_{s}^{j})}_{L^{p}(\Omega)^{3 \times 3}}.
\end{aligned}\]
Therefore, choosing simultaneously $j \in \mathbb{N}$ sufficiently large and $T$ sufficiently small, we can conclude the result.
\end{proof}

We give a short summary of the conditions that provide the validity of~\eqref{conv-22}.

\begin{remark}\label{rem:smallness}
\begin{itemize}
\item[(a)] By virtue of Remark~\ref{rem:CT}, we conclude in the case $T = + \infty$ that if $\norm{a}_{L^{p}_{\sigma}(\Omega)} +\norm{\nabla b}_{L^{p}(\Omega)^{3 \times 3}}$ is small enough, the validity of~\eqref{conv-22} can be inferred by Lemma~\ref{lemma:initialvalues}. This implies eventually the global existence under a suitable smallness condition on the initial data $a$ and $\nabla b$.
\item[(b)] If $p > 3$, then~\eqref{conv-22} follows by using the previous Lemma~\ref{lemma:initialvalues} and the fact that $\lim_{T \rightarrow 0} \tilde{C}_{T} = 0$, cf.\@ Remark~\ref{rem:CT}. This will imply local existence of solutions without any smallness assumptions on the initial data.
\item[(c)] If $p =3 $, then~\eqref{conv-22} follows from Lemma~\ref{small} for small times $T$. Again this will imply local existence of solutions without any smallness assumption on the initial data.
\end{itemize}
\end{remark}

\subsubsection{Continuity with respect to time}
We shall next prove the continuity with respect to time of $u_{j}$, $y_{j}$, $\nabla y_{j}$, and $x_{j}$ considered earlier. In this direction, let $3 \leq \max \{p,\frac{q}{2}\} \leq r \leq q < 3+\eps $ and we shall begin our consideration with the sequence $(u_{j})_{j \in \mathbb{N}} $. \par
Let $t_{0} \in (0,T)  $. Then, for $h > 0$ small enough, we have
\begin{equation}
\begin{aligned}
&\norm{u_{0}(t_{0}+h)-u_{0}(t_{0})}_{L^{r}_{\sigma}(\Omega)} = \norm{e^{-(t_{0}+h)A}a-e^{-t_{0}A}a}_{L^{r}_{\sigma}(\Omega)} \leq C t_{0}^{-\frac{3}{2}(\frac{1}{p}-\frac{1}{r})} \norm{e^{-hA}a-a}_{L^{p}_{\sigma}(\Omega)}
\end{aligned}
\label{cont25}
\end{equation}
which tends to $0$ as $h \rightarrow 0$ by strong continuity of $(e^{-tA})_{t \geq 0} $ on $L^{p}_{\sigma}(\Omega) $. Similarly, 
\begin{equation}
\begin{aligned}
\norm{u_{0}(t_{0}-h)-u_{0}(t_{0})}_{L^{r}_{\sigma}(\Omega)}&= \norm{e^{-(t_{0}-h)A}a-e^{-t_{0}A}a}_{L^{r}_{\sigma}(\Omega)}= \norm{e^{-(t_{0}-h)A} [Id-e^{-hA}] a}_{L^{r}_{\sigma}(\Omega)} \\
&\leq C (t_{0}-h)^{-\frac{3}{2}(\frac{1}{p}-\frac{1}{r})} \norm{e^{-hA}a-a}_{L^{p}_{\sigma}(\Omega)},
\end{aligned}
\label{cont26}
\end{equation}
which tends to $0$ as $h \rightarrow 0$. \par
Next, let $h$ be small enough so that $0<h<T-t_{0} $. We want to show that
\begin{equation}
\begin{aligned}
&\Big\| \int_{0}^{t_{0}+h} e^{-(t_{0}+h-s)A} F_{u}(u_{j},\nabla y_{j}) \ \d s - \int_{0}^{t_{0}} e^{-(t_{0}-s)A} F_{u}(u_{j},\nabla y_{j}) \ \d s \Big\|_{L^{r}_{\sigma}(\Omega)}
\end{aligned}
\label{cont30}
\end{equation}
converges to $0$ as $h \rightarrow 0$. To see this, we note that~\eqref{cont30} can be majorised by
\begin{equation}
\begin{aligned}
&\int_{t_{0}}^{t_{0}+h} \norm{e^{-(t_{0}+h-s)A} F_{u}(u_{j},\nabla y_{j})}_{L^{r}_{\sigma}(\Omega)} \ \d s + \int_{0}^{t_{0}} \norm{[e^{-hA}-Id] e^{-(t_{0}-s)A} F_{u}(u_{j},\nabla y_{j})}_{L^{r}_{\sigma}(\Omega)} \ \d s.
\end{aligned}
\label{cont31}
\end{equation}
The first term in~\eqref{cont31} can be dominated as in~\eqref{conv-100} by
\begin{equation}
\begin{aligned}
&C \int_{t_{0}}^{t_{0}+h} e^{-(t_{0}+h)\omega} (t_{0}+h-s)^{-\frac{1}{2}-\frac{3}{2}(\frac{2}{q}-\frac{1}{r})} s^{-3(\frac{1}{p}-\frac{1}{q})} \ \d s \Big[ k^{u}_{j}(T)^2+k^{\nabla y}_{j}(T)^2 \Big] \\
&\leq C e^{-(t_{0}+h)\omega} (t_{0}+h)^{\frac{1}{2}-3(\frac{1}{p}-\frac{1}{2r})} \int_{\frac{t_{0}}{t_{0}+h}}^{1} (1-\tilde{s})^{-\frac{1}{2}-\frac{3}{2}(\frac{2}{q}-\frac{1}{r})} \tilde{s}^{-3(\frac{1}{p}-\frac{1}{q})} \ \d \tilde{s} \Big[k^{u}_{j}(T)^2+k^{\nabla y}_{j}(T)^2 \Big],
\end{aligned}
\end{equation}
which converges to $0$ since $\frac{t_{0}}{t_{0}+h} \rightarrow 1 $ as $h \rightarrow 0$. \par
To deal with the second term of~\eqref{cont31}, we note that
\begin{equation}
\norm{e^{-(t_{0}-s)A} F_{u}(u_{j},\nabla y_{j})}_{L^{r}_{\sigma}(\Omega)} \leq C e^{-t_{0}\omega} (t_{0}-s)^{-\frac{1}{2}-\frac{3}{2}(\frac{2}{q}-\frac{1}{r})} s^{-3(\frac{1}{p}-\frac{1}{q})} \Big[ k^{u}_{j}(T)^2+k^{\nabla y}_{j}(T)^2 \Big],
\notag
\end{equation}
which is integrable as seen in~\eqref{conv-100}. Thus, by the dominated convergence theorem together with the strong continuity of the semigroup this term converges to zero as well. The left-continuity can be proved in a similar manner. Consequently, we have for $3 \leq \max \{p,\frac{q}{2}\} \leq r \leq q < 3+\eps $, 
\begin{align*}
 e^{\frac{\omega s}{2}} s^{\frac{3}{2}(\frac{1}{p}-\frac{1}{r})} u_{j} \in BC((0,T);L^{r}_{\sigma}(\Omega)),
\end{align*}
the boundedness being already proved in~\eqref{conv-100}. Next, if $r > p $, then~\eqref{conv-100},~\eqref{Eq: Boundedness of sequences},~\eqref{Eq: Definition of bounds}, and Lemma~\ref{small} imply
\begin{align*}
 \lim_{s \rightarrow 0} e^{\frac{\omega s}{2}} s^{\frac{3}{2}(\frac{1}{p}-\frac{1}{r})} \norm{u_{j}(s)}_{L^r_{\sigma}(\Omega)}=0,
\end{align*}
which yields $ e^{\frac{\omega s}{2}} s^{\frac{3}{2}(\frac{1}{p}-\frac{1}{r})} u_{j} \in BC([0,T);L^{r}_{\sigma}(\Omega))$. Finally, if $r = p$, choose $q$ close enough to $p$ such that $q \leq 2p$, then one derives similarly to~\eqref{conv-100} the estimate
\begin{align*}
 \sup_{0<t<T} e^{\frac{\omega t}{2}} \norm{u_{j} - u_0}_{L_{\sigma}^p(\Omega)} \leq C k_{j - 1}^q (T)^2
\end{align*}
for some constant $C > 0$. This, combined with~\eqref{Eq: Boundedness of sequences},~\eqref{Eq: Definition of bounds}, Lemma~\ref{small}, and the strong continuity of the Stokes semigroup proves
\begin{align*}
 \lim_{s \to 0} \| u_j (s) - a \|_{L^p_{\sigma} (\Omega)} = 0
\end{align*}
and thus, $u_j \in BC ([0 , T) ; L^p_{\sigma} (\Omega))$. \par
We shall next consider the continuity for $\nabla y_{j} $. Let $t_{0} \in (0,T)$ and $h > 0$ be small enough. Then,
\begin{equation}
\begin{aligned}
&\norm{\nabla y_{0}(t_{0}+h)-\nabla y_{0}(t_{0})}_{L^{r}(\Omega)^{3 \times 3}} \leq C t_{0}^{-\frac{1}{2}-\frac{3}{2}(\frac{1}{p}-\frac{1}{r})} \norm{e^{-hB}b_{s}-b_{s}}_{L^{p}(\Omega)^3},
\end{aligned}
\label{cont34}
\end{equation}
which converges to $0$ as $h \rightarrow 0$, by the strong continuity of the semigroup. The left-continuity follows similarly. \par
Now, let $h$ be small enough so that $0<h<T-t_{0} $. Then 
\begin{equation}
\begin{aligned}
\Big\| \int_{0}^{t_{0}+h} \nabla e^{-(t_{0}+h-s)B} &F_{y}(u_{j},\nabla y_{j},y_{j},x_{j},\overline{b}) \ \d s - \int_{0}^{t_{0}} \nabla e^{-(t_{0}-s)B} F_{y}(u_{j},\nabla y_{j},y_{j},x_{j},\overline{b}) \ \d s \Big\|_{L^{r}(\Omega)^{3 \times 3}}\\
&\leq \int_{t_{0}}^{t_{0}+h} \norm{\nabla e^{-(t_{0}+h-s)B} F_{y}(u_{j},\nabla y_{j},y_{j},x_{j},\overline{b})}_{L^{r}(\Omega)^{3 \times 3}}\ \d s \\
 &\qquad\qquad + \int_{0}^{t_{0}} \norm{\nabla [e^{-hB}-Id]e^{-(t_{0}-s)B} F_{y}(u_{j},\nabla y_{j},y_{j},x_{j},\overline{b})}_{L^{r}(\Omega)^{3 \times 3}} \ \d s.
\end{aligned}
\label{cont35}
\end{equation}
We note that by Proposition~\ref{prop:smoothing}
\begin{equation}
\begin{aligned}
 \int_{t_{0}}^{t_{0}+h} &\norm{\nabla e^{-(t_{0}+h-s)B} F_{y}(u_{j},\nabla y_{j},y_{j},x_{j},\overline{b})}_{L^{r}(\Omega)^{3 \times 3}}\ \d s \\
&\leq C e^{-(t_{0}+h)\omega} (t_{0}+h)^{\frac{1}{2}-3(\frac{1}{p}-\frac{1}{2r})} \int_{\frac{t_{0}}{t_{0}+h}}^{1} (1-\tilde{s})^{-\frac{1}{2}-\frac{3}{2}(\frac{2}{q}-\frac{1}{r})} \tilde{s}^{-3(\frac{1}{p}-\frac{1}{q})} \ \d \tilde{s} \\
 &\qquad\qquad\qquad\qquad\qquad\qquad\qquad\qquad\qquad  \Big[k^{u}_{j}(T) k^{\nabla y}_{j}(T) + k^{\nabla y}_{j}(T)^2 (k^{x}_{j}(T)+k^{y}_{j}(T)+ \lvert \overline{b} \rvert) \Big]
\end{aligned}
\notag
\end{equation}
and this converges to $0$ as $h \rightarrow 0$ since $\frac{t_{0}}{t_{0}+h} \rightarrow 1 $. To deal with the second term, we estimate by virtue of Proposition~\ref{Prop: Square root properties}~(b)
\begin{equation}
\begin{aligned}
\|\nabla [e^{-hB}-Id]e^{-(t_{0}-s)B} &F_{y}(u_{j},\nabla y_{j},y_{j},x_{j},\overline{b})\|_{L^{r}(\Omega)^{3 \times 3}}  \\
 &\leq C \norm{[e^{-hB}-Id]B^{\frac{1}{2}} e^{-(t_{0}-s)B} F_{y}(u_{j},\nabla y_{j},y_{j},x_{j},\overline{b})}_{L^{r}(\Omega)^3}.
\end{aligned}
\notag
\end{equation}
Now, by Propositions~\ref{Prop: Square root properties} and~\ref{prop:smoothing} and standard estimates for fractional powers applied to analytic semigroups
\begin{equation}
\norm{B^{\frac{1}{2}} e^{-(t_{0}-s)B} F_{y}(u_{j},\nabla y_{j},y_{j},x_{j},\overline{b})}_{L^{r}(\Omega)^3} \leq C e^{-(t_{0}-s)\omega} (t_{0}-s)^{-\frac{1}{2}-\frac{3}{2}(\frac{2}{q}-\frac{1}{r})} \norm{F_{y}(u_{j},\nabla y_{j},y_{j},x_{j},\overline{b})}_{L^{\frac{q}{2}}(\Omega)^3}
\notag
\end{equation}
and
\begin{equation}
\begin{aligned}
 \int_{0}^{t_{0}} &e^{-(t_{0}-s)\omega} (t_{0}-s)^{-\frac{1}{2}-\frac{3}{2}(\frac{2}{q}-\frac{1}{r})} \norm{F_{y}(u_{j},\nabla y_{j},y_{j},x_{j},\overline{b})}_{L^{\frac{q}{2}}(\Omega)^3} \ \d s \\
&\leq C \int_{0}^{t_{0}} (t_{0}-s)^{-\frac{1}{2}-\frac{3}{2}(\frac{2}{q}-\frac{1}{r})} s^{-3(\frac{1}{p}-\frac{1}{q})} \ \d s \Big[k^{u}_{j}(T) k^{\nabla y}_{j}(T) + k^{\nabla y}_{j}(T)^2 (k^{x}_{j}(T)+k^{y}_{j}(T)+ \lvert \overline{b} \rvert) \Big]
\end{aligned}
\notag
\end{equation}
which is finite. Hence, we can use the dominated convergence theorem together with the strong continuity of the semigroup to infer that the second term in~\eqref{cont35} converges to zero as well, thus proving right-continuity.  The left-continuity can be studied in a similar fashion. Therefore, for $3 \leq \max \{p,\frac{q}{2}\} \leq r \leq q < 3+\eps$, we find by~\eqref{conv-101}
\begin{align*}
&e^{\frac{\omega s}{2}} s^{\frac{3}{2}(\frac{1}{p}-\frac{1}{r})} \nabla y_{j} \in BC((0,T);L^{r}(\Omega)^{3 \times 3}).
\end{align*}
If $r > p$, then~\eqref{conv-101},~\eqref{Eq: Boundedness of sequences},~\eqref{Eq: Definition of bounds}, Lemma~\ref{lemma:initialvalues}, and Lemma~\ref{small} imply
\begin{align*}
 \lim_{s \rightarrow 0} e^{\frac{\omega s}{2}} s^{\frac{3}{2}(\frac{1}{p}-\frac{1}{r})} \norm{\nabla y_{j}(s)}_{L^r(\Omega)^{3 \times 3}}=0,
\end{align*}
yielding $e^{\frac{\omega s}{2}} s^{\frac{3}{2}(\frac{1}{p}-\frac{1}{r})} \nabla y_{j} \in BC([0,T);L^{r}(\Omega)^{3 \times 3})$. In the case $p = r$ choose $q$ close enough to $p$ such that $q \leq 2 p$. Then, a similar calculation to~\eqref{conv-101} yields
\begin{align*}
 &\sup_{0<t<T} e^{\frac{\omega t}{2}} t^{\frac{3}{2}(\frac{1}{p}-\frac{1}{r})} \norm{\nabla [y_{j} - y_0]}_{L^r(\Omega)^{3 \times 3}} \leq C \Big[k^q_{j - 1}(T)^2 + k^{q}_{j - 1}(T)^2 (k^{\infty}_{j - 1}(T) + \lvert \overline{b} \rvert) \Big],
\end{align*}
for some constant $C > 0$. This, combined with~\eqref{Eq: Boundedness of sequences},~\eqref{Eq: Definition of bounds}, Lemma~\ref{lemma:initialvalues}, Lemma~\ref{small}, and Proposition~\eqref{Prop: Square root properties}
\begin{align*}
 \lim_{s \to 0} \| \nabla [y_j (s) - b] \|_{L^p(\Omega)^{3 \times 3}} \leq \lim_{s \to 0} \| \nabla [y_j (s) - y_0 (s)] \|_{L^p(\Omega)^{3 \times 3}} + \lim_{s \to 0} \| B^{\frac{1}{2}} [e^{- s B} b - b] \|_{L^p (\Omega)^{3 \times 3}} = 0.
\end{align*}
Altogether, it follows $\nabla y_j \in BC ([0 , T) ; L^p(\Omega)^{3 \times 3})$. \par
Finally, the continuity of $y_j$ with respect to the $L^{\infty}$-norm and the continuity of $x_j$ follows by similar calculations, which are omitted here.

\subsubsection{} 
Summing up the discussions above, and using the fact that by~\eqref{Eq: Contraction of delta} the sequence $(\delta_{j})_{j \in \mathbb{N}} $ gives rise to a contraction, we conclude the convergence of the sequences $(u_{j})_{j \in \mathbb{N}}$, $(y_{j})_{j \in \mathbb{N}}$, and $(x_{j})_{j \in \mathbb{N}}$ to functions $u$, $y$, and $x  $ such that 
\begin{equation}
\begin{aligned}
&u \in S^{u}_{q}(T),\quad d:=y+x \in  BC([0,T);L^{\infty}(\Omega)^3), \quad \text{and} \quad \nabla d \in S^{d}_{q}(T),\\
&\lim_{s \rightarrow 0} e^{\frac{\omega s}{2}} s^{\frac{3}{2}(\frac{1}{p}-\frac{1}{q})} \norm{u(s)}_{L^q_{\sigma}(\Omega)}=0, 
\qquad \lim_{s \rightarrow 0} e^{\frac{\omega s}{2}} s^{\frac{3}{2}(\frac{1}{p}-\frac{1}{q})} \norm{\nabla d(s)}_{L^q(\Omega)^{3 \times 3}}=0.
\end{aligned}
\label{conv-200}
\end{equation}
Here, $q$ is any number satisfying $p < q < 3 + \eps$. If $q$ is close enough to $p$, i.e., if $q \leq 2 p$, then one can derive similar inequalities to~\eqref{conv-100} and~\eqref{conv-101} for the differences of $u_j$ and $u_l$ and of $\nabla y_j$ and $\nabla y_l$ for $j , l \in \N$ and for $r = p$. This proves that $(u_j)_{j \in \N}$ is a Cauchy sequence in $BC ([0 , T) ; L^p_{\sigma} (\Omega))$ and that $(\nabla y_j)_{j \in \N}$ is a Cauchy sequence in $BC ([0 , T) ; L^p(\Omega)^{3 \times 3})$ whenever 
\begin{align*}
 \big(e^{\frac{\omega t}{2}} t^{\frac{3}{2} (\frac{1}{p} - \frac{1}{q})} u_j \big)_{j \in \N} \subset BC ([0 , T) ; L^q_{\sigma} (\Omega)),& \quad \big(e^{\frac{\omega t}{2}} t^{\frac{3}{2} (\frac{1}{p} - \frac{1}{q})} \nabla y_j \big)_{j \in \N} \subset BC ([0 , T) ; L^q (\Omega)^{3 \times 3}),\\
 \big( y_j \big)_{j \in \N} \subset BC ([0 , T) ; L^{\infty} (\Omega)^3),& \quad \text{and} \quad \big( x_j \big)_{j \in \N} \subset BC ([0 , T) ; \R^3)
\end{align*}
are Cauchy sequences. Consequently, $u \in BC([0,T);L^{p}_{\sigma}(\Omega))$ and $d \in BC([0,T);W^{1,p}(\Omega)^3)$ and thus $u$ and $d$ give rise to a mild solution satisfying properties (a), (b), and $(c)$ of Theorem~\ref{thm:main}.

\subsubsection{Uniqueness}
Let $u_{1}, u_{2} \in BC([0,T);L^{p}_{\sigma}(\Omega)) $ and $d_{1}, d_{2} \in BC([0,T);W^{1,p}(\Omega)^3)$ be such that $(u_{1},d_{1}) $ and $(u_{2},d_{2}) $ are two mild solutions satisfying~\eqref{conv-200} for some $q\in (p, 3 + \eps)$. 
\par
We consider the differences $u_{1}-u_{2} $ and $d_{1}-d_{2} $ and proceed as in the estimation of the sequence $(\delta_{j})_{j \in \N}$ to infer that an analogous version of~\eqref{Eq: Contraction of delta} is valid for the differences. Now, on the left-hand side as well as on the right-hand side of this analogous version of~\eqref{Eq: Contraction of delta} the differences of $u_1$ and $u_2$ and of $d_1$ and $d_2$ appear. This already implies the uniqueness. The only major point to note in this regard is that if $p >3 $, the conditions 
\begin{equation}
\begin{aligned}
&\lim_{s \rightarrow 0} e^{\frac{\omega s}{2}} s^{\frac{3}{2}(\frac{1}{p}-\frac{1}{q})} \norm{u_{i}(s)}_{L^q_{\sigma}(\Omega)}=0,\\
&\lim_{s \rightarrow 0} e^{\frac{\omega s}{2}} s^{\frac{3}{2}(\frac{1}{p}-\frac{1}{q})} \norm{\nabla d_{i}(s)}_{L^q(\Omega)^{3 \times 3}}=0, \ i=1,2,
\end{aligned}
\label{conv-201}
\end{equation}
are not required since we can use the fact that the constant $\tilde{C}_{T} \rightarrow 0 $ as $T \rightarrow 0 $.
But for $p=3$, the conditions~\eqref{conv-201} need to be assumed in addition.

\subsection{Retrieving the $\lvert d \rvert = 1$ condition}
\label{Subsec: Retrieving}
Let $u$ and $d$ be mild solutions to~\eqref{eq:PLCD} with Neumann boundary conditions for $d$. Moreover, assume that the initial conditions satisfy $a \in L^p_{\sigma} (\Omega)$ and $b \in W^{1 , p} (\Omega)^3$ with $\lvert b \rvert = 1$ in $\Omega$ for some $3 \leq p < 3 + \eps$. Without loss of generality let $T < \infty$ and assume further, that for every $p \leq q < 3 + \eps$,
\begin{align*}
 t \mapsto e^{\frac{\omega t}{2}} t^{\frac{3}{2} (\frac{1}{p} - \frac{1}{q})} u (t) &\in B C ([0 , T) ; L^q_{\sigma} (\Omega)), \\
 t \mapsto e^{\frac{\omega t}{2}} t^{\frac{3}{2} (\frac{1}{p} - \frac{1}{q})} \nabla d (t) &\in B C ([0 , T) ; L^q (\Omega)^{3 \times 3}), \\
 d &\in B C ([0 , T) ; L^{\infty} (\Omega)^3).
\end{align*}
These are precisely the properties of the solutions constructed in Subsection~\ref{Subsec: Existence and uniqueness}. Especially, the second equation of~\eqref{eq:LCD} shows that $d$ is a solution to the linear heat equation with right-hand side
\begin{align*}
 - (u \cdot \nabla) d + \lvert \nabla d \rvert^2 d \in L^s(0 , T ; L^{\frac{p}{2}} (\Omega)^3)
\end{align*}
and with initial value $b \in W^{1 , p} (\Omega)^3 \subset (L^{\frac{p}{2}} (\Omega)^3, \dom(B_{\frac{p}{2}}))_{1 - 1 / s, s}$ for every $1 < s < 2$. The inclusion follows from the following observations. First of all, $W^{1 , p} (\Omega)^3 \subset W^{1 , \frac{p}{2}} (\Omega)^3$ is a consequence of H\"older's inequality. Next, the equality $\dom(B_{p / 2}^{\frac{1}{2}}) = W^{1 , \frac{p}{2}} (\Omega)^3$ follows from Proposition~\ref{Prop: Square root properties} and finally $\dom(B^{\frac{1}{2}}) \subset (X , \dom(B))_{1 / 2 , \infty}$ is valid for every sectorial operator $B$ of angle less than $\pi / 2$, see~\cite[Cor.~6.6.3]{Haase}. The embedding $(X , Y)_{1 / 2 , \infty} \subset (X , Y)_{1 - 1 / s , s}$ is a standard embedding in the theory of real interpolation whenever $Y \subset X$. In our case $X=L^{\frac{p}{2}} (\Omega)^3$ and $Y = \dom(B_{\frac{p}{2}})$. \par
Now, $B$ has maximal regularity, cf.~Proposition~\ref{Prop: Maximal regularity of strong operators}. This means that whenever $\mathfrak{d}$ is a mild solution to the equation
\begin{align*}
\left\{ \begin{aligned}
 \mathfrak{d}^{\prime} +B \mathfrak{d} &= f \qquad (0 < t < T) \\
 \mathfrak{d} (0) &= a_{\mathfrak{d}}
\end{aligned} \right.
\end{align*}
with $f \in L^s (0 , T ; L^{
\frac{p}{2}} (\Omega)^3)$ and $a_{\mathfrak{d}} \in (L^{\frac{p}{2}} (\Omega)^3 , \dom(B_{\frac{p}{2}}))_{1 - 1 / s , s}$, then $\mathfrak{d}$ is already a strong solution, i.e.,
\begin{align*}
 \mathfrak{d} \in W^{1 , s} (0 , T ; L^{\frac{p}{2}} (\Omega)^3) \cap L^s (0 , T ; \dom(B_{\frac{p}{2}}))
\end{align*}
and $\mathfrak{d}$ solves the equation above pointwise almost everywhere. \par
As we have shown, for $1 < s < 2$ the function $d$ solves exactly such an equation so that by maximal regularity $d$ lies in the maximal regularity space stated in the previous equation. Especially, $d^{\prime} (t)$ exists for almost every $t \in (0 , T)$ and $d(t) \in \dom(B_{\frac{p}{2}})$ for almost every $t \in (0 , T)$. Moreover, the second equation of~\eqref{eq:LCD} holds for almost every $t \in (0 , T)$ in the strong sense. More precisely, this means
\begin{align*}
 d^{\prime} (t) +B d(t) = - (u(t) \cdot  \nabla) d(t) + \lvert \nabla d(t) \rvert^2 d(t) \qquad \text{a.e. } t \in (0 , T).
\end{align*}
The sense in which this equality has to be read is the following. As the Laplacian is defined via a sesquilinear form and since $\dom(B_{\frac{p}{2}}) \subset W^{1 , p / 2} (\Omega)^3$ by Proposition~\ref{Prop: Square root properties} it holds for every $\vartheta \in C^{\infty} (\overline{\Omega})^3$
\begin{align}
\label{Eq: Tested d-equation}
 \int_{\Omega} d^{\prime} (t) \cdot \overline{\vartheta} \; \d x + \int_{\Omega} \nabla d(t) \cdot \overline{\nabla \vartheta} \; \d x = - \int_{\Omega} (u (t) \cdot \nabla) d (t) \cdot \overline{\vartheta} \; \d x + \int_{\Omega} \lvert \nabla d (t) \rvert^2 d (t) \cdot \overline{\vartheta} \; \d x.
\end{align}
Clearly, by density this identity holds for all $\vartheta \in W^{1 , (\frac{p}{2})^{\prime}} (\Omega)^3$. Next, define
\begin{align*}
 \varphi := \lvert d \rvert^2 - 1.
\end{align*}
The aim is to show that under the present conditions, $\varphi$ is identically zero. To do so, we derive an equation for $\varphi$ by employing~\eqref{Eq: Tested d-equation}. However, note first that by assumption $\varphi(t) \in L^{\infty} (\Omega)$ for every $t \in (0 , T)$ and that
\begin{align}
\label{Eq: Derivatives of phi}
 \partial_k \varphi (t) = 2 \overline{d (t)} \cdot \partial_k d (t) \in L^p (\Omega) \quad \text{and} \quad \varphi^{\prime} (t) = 2 \overline{d (t)} \cdot d^{\prime} (t) \in L^{\frac{p}{2}} (\Omega).
\end{align}
Thus, let $w$ be a test function in $C^{\infty} (\overline{\Omega})$. Then, by~\eqref{Eq: Derivatives of phi} and the product rule
\begin{align*}
 \int_{\Omega} \varphi^{\prime} \overline{w} \; \d x + \int_{\Omega} \nabla \varphi \cdot \overline{\nabla w} \; \d x = 2 \int_{\Omega} d^{\prime} \cdot \overline{(d w)} \; \d x + 2 \sum_{k , l = 1}^3 \int_{\Omega} \partial_k d_l \overline{\partial_k (d_l w)} \; \d x - 2 \sum_{k , l = 1}^3 \int_{\Omega} \partial_k d_l \overline{(\partial_k d_l) w} \; \d x.
\end{align*}
Now, we use~\eqref{Eq: Tested d-equation} with $d w$ as the test function and the fact that $\varphi = \lvert d \rvert^2 - 1$. Note that this is possible as $d(t) \in L^{\infty} (\Omega)^3$ and since $p \geq 3$, $\nabla d(t) \in L^p(\Omega)^{3 \times 3} \subset L^{\frac{p}{p - 2}} (\Omega)^{3 \times 3} = L^{(\frac{p}{2})^{\prime}} (\Omega)^{3 \times 3}$. Consequently, it holds
\begin{align*}
 \int_{\Omega} \varphi^{\prime} \overline{w} \; \d x + \int_{\Omega} \nabla \varphi \cdot \overline{\nabla w} \; \d x = - 2 \int_{\Omega} (u \cdot \nabla) d \cdot \overline{(d w)} \; \d x + 2 \int_{\Omega} \lvert \nabla d \rvert^2 \varphi \overline{w} \; \d x.
\end{align*}
Finally, note that
\begin{align*}
 (u \cdot \nabla) d \cdot \overline{d} = \sum_{k , l = 1}^3 u_k \partial_k d_l \overline{d_l} = \frac{1}{2} u \cdot \nabla \lvert d \rvert^2 = \frac{1}{2}u \cdot \nabla \varphi,
\end{align*}
so that
\begin{align}
\label{Eq: Equation for phi}
 \int_{\Omega} \varphi^{\prime} \overline{w} \; \d x + \int_{\Omega} \nabla \varphi \cdot \overline{\nabla w} \; \d x = - \int_{\Omega} u \cdot \nabla \varphi \overline{w} \; \d x + 2 \int_{\Omega} \lvert \nabla d \rvert^2 \varphi \overline{w} \; \d x
\end{align}
holds true for all $w \in C^{\infty} (\overline{\Omega})$. Since $u(t) \in L^p(\Omega)^3, \varphi(t) \in L^{\infty} (\Omega)$, and $\nabla d (t) \in L^p(\Omega)^{3 \times 3}$,~\eqref{Eq: Derivatives of phi} implies
\begin{align*}
 \varphi^{\prime} (t) , \quad \nabla \varphi(t) , \quad u (t) \cdot \nabla \varphi (t) , \quad \lvert \nabla d \rvert^2 \varphi \quad \in L^{\frac{p}{2}} (\Omega).
\end{align*}
Summarising, by density~\eqref{Eq: Equation for phi} remains valid for $w \in W^{1 , (\frac{p}{2})^{\prime}} (\Omega)$. Moreover, for $p \geq 3$,~\eqref{Eq: Derivatives of phi} implies also that for almost every $t \in (0 , T)$ we have $\varphi(t) \in W^{1 , p} (\Omega) \subset W^{1 , (\frac{p}{2})^{\prime}} (\Omega)$. Thus, $\varphi$ itself is an admissible test function so that~\eqref{Eq: Equation for phi} turns into
\begin{align}
\label{Eq: phi-equation tested with phi}
 \int_{\Omega} \varphi^{\prime} \varphi \; \d x + \int_{\Omega} \lvert \nabla \varphi \rvert^2 \; \d x = - \int_{\Omega} u \cdot \nabla \varphi \varphi \; \d x + 2 \int_{\Omega} \lvert \nabla d \rvert^2 \varphi^2 \; \d x.
\end{align}
Moreover, note that
\begin{align*}
 \int_0^t \int_{\Omega} \varphi^{\prime} \varphi \; \d x \; \d s = \frac{1}{2} \int_{\Omega} \varphi (t)^2 \; \d x
\end{align*}
since $\varphi(0) = 0$ and that
\begin{align*}
 \int_{\Omega} u \cdot \varphi \nabla \varphi  \; \d x = \frac{1}{2} \int_{\Omega} u \cdot \nabla \varphi^2 \; \d x = 0
\end{align*}
since $u$ is divergence free and vanishes on the boundary. After integrating~\eqref{Eq: phi-equation tested with phi} with respect to time one finds
\begin{align}
\label{Eq: Energy of phi}
 \frac{1}{2} \int_{\Omega} \varphi (t)^2 \; \d x + \int_0^t \int_{\Omega} \lvert \nabla \varphi(s) \rvert^2 \; \d x \; \d s = 2 \int_0^t \int_{\Omega} \lvert \nabla d(s) \rvert^2 \varphi(s)^2 \; \d x \; \d s.
\end{align}
To show that~\eqref{Eq: Energy of phi} implies that $\varphi$ is zero we show that Gronwall's lemma is applicable to the function $\varphi^2$. To do so, let $t_0 \in (0 , t)$ be a number to be determined and split the integral on the right-hand side of~\eqref{Eq: Energy of phi} as
\begin{align*}
 \int_0^t \int_{\Omega} \lvert \nabla d(s) \rvert^2 \varphi (s)^2 \; \d x \; \d s = \int_0^{t_0} \int_{\Omega} \lvert \nabla d(s) \rvert^2 \varphi (s)^2 \; \d x \; \d s + \int_{t_0}^t \int_{\Omega} \lvert \nabla d(s) \rvert^2 \varphi (s)^2 \; \d x \; \d s =: \mathrm{I} + \mathrm{II}.
\end{align*}

We estimate the term $\mathrm{II}$ first:
For $p < q < 3 + \eps$ one estimates by means of H\"older's and Sobolev's inequality as well as the decay estimates of $\nabla d$,
\begin{align*}
 \int_{t_0}^t \int_{\Omega} \lvert \nabla d(s) \rvert^2 \varphi(s)^2 \; \d x \; \d s &\leq \int_{t_0}^t \| \nabla d(s) \|^2_{L^q} \| \varphi(s) \|_{L^{2 \cdot (\frac{q}{2})^{\prime}}}^2 \; \d s \leq C t_0^{- 3 (\frac{1}{p} - \frac{1}{q})} \int_{t_{0}}^t \| \varphi (s) \|_{L^2}^{2 - 2 \alpha} \| \varphi (s) \|_{W^{1 , 2}}^{2 \alpha} \; \d s,
\end{align*}
where $\alpha = 3 / q$. Continuing the estimate above with Young's inequality delivers
\begin{align*}
 C t_0^{- 3 (\frac{1}{p} - \frac{1}{q})} \int_{t_0}^t \| \varphi (s) \|_{L^2}^{2 - 2 \alpha} \| \varphi (s) \|_{W^{1 , 2}}^{2 \alpha} \; \d s \leq C(t_0 , q) \int_0^t \| \varphi(s) \|_{L^2}^2 \; \d s + \frac{1}{8} \int_0^t \int_{\Omega} \lvert \nabla \varphi (s) \rvert^2 \; \d x \; \d s.
\end{align*}

Next, we estimate the term $\mathrm{I}$: Recall that the initial value for $d$ is denoted by $b$. Let $(b_{n})_{n \in \N} \subset C^{\infty} (\overline{\Omega})^3$ be such that $b_{n} \to b$ in $W^{1 , p} (\Omega)^3$. Using the triangle inequality, H\"older's inequality, and Sobolev's embedding in a row, we estimate for $p < q < 3 + \eps$
\begin{align*}
 \int_0^{t_0} \int_{\Omega} &\lvert \nabla d(s) \rvert^2 \varphi (s)^2 \; \d x \; \d s \\
 &\leq 2 \int_0^{t_0} \int_{\Omega} \lvert \nabla [d(s) - b_{n}] \rvert^2 \varphi (s)^2 \; \d x \; \d s + 2 \int_0^{t_0} \int_{\Omega} \lvert \nabla b_{n} \rvert^2 \varphi (s)^2 \; \d x \; \d s \\
 &\leq 2 \int_0^{t_0} \| \nabla [d(s) - b_{n}] \|_{L^p}^2 \| \varphi (s) \|_{L^{2 \cdot (p / 2)^{\prime}}}^2 \; \d s 
 + \| \nabla b_{n} \|_{L^q}^2 \int_0^{t_0} \| \varphi (s) \|_{L^{2 \cdot (q / 2)^{\prime}}}^2 \; \d s \\
 &\leq C \sup_{0 < s < t_0} \| \nabla [d(s) - b_{n}] \|_{L^p}^2 \int_0^{t_0} \| \varphi(s) \|_{W^{1 , 2}}^2 \; \d s 
 + C \| \nabla b_{n} \|_{L^q}^2 \int_0^{t_0} \| \varphi(s) \|_{L^2}^{2 - 2 \alpha} \| \varphi(s) \|_{W^{1 , 2}}^{2 \alpha} \; \d s,
\end{align*}
where $\alpha$ is chosen as in the estimate of term $\mathrm{II}$. Employing Young's inequality implies
\begin{align*}
 \int_0^{t_0} \int_{\Omega} \lvert \nabla d(s) \rvert^2 \varphi (s)^2 \; \d x &\leq \big\{ C \sup_{0 < s < t_0} \| \nabla [d(s) - b_{ n}] \|_{L^p}^2 + C(q , \| \nabla b_{ n} \|_{L^q}^2) \big\} \int_0^t \| \varphi(s) \|_{L^2}^2 \; \d s \\
 &\qquad+ \Big\{ \frac{1}{16} + C \sup_{0 < s < t_0} \| \nabla [d(s) - b_{ n}] \|_{L^p}^2 \Big\} \int_0^t \| \nabla \varphi (s) \|_{L^2}^2 \; \d s.
\end{align*}
Finally,
\begin{align*}
 \sup_{0 < s < t_0} \| \nabla [d(s) - b_{n}] \|_{L^p}^2 \leq 2 \sup_{0 < s < t_0} \| \nabla [d(s) - b] \|_{L^p}^2 + 2 \| \nabla [b - b_{ n}] \|_{L^p}^2.
\end{align*}
Choose $t_0$ small enough, such that
\begin{align*}
 2 \sup_{0 < s < t_0} \| \nabla [d(s) - b] \|_{L^p}^2 < \frac{1}{32 C}
\end{align*}
and $n$ large enough, such that
\begin{align*}
 2 \| \nabla [b - b_{ n}] \|_{L^p}^2 < \frac{1}{32 C}.
\end{align*}
For these fixed numbers $t_0$ and $n$, we finally find
\begin{align*}
 \int_0^t \int_{\Omega} \lvert \nabla d(s) \rvert^2 &\varphi (s)^2 \; \d x \; \d s \\
 &\leq 2 \bigg\{ \big\{ C(t_0 , q) + C \sup_{0 < s < t_0} \| \nabla [d(s) - b_{ n}] \|_{L^p}^2 + C(q , \| \nabla b_{n} \|_{L^q}^2) \big\} \int_0^t \| \varphi (s) \|_{L^2}^2 \; \d s \\
 &\qquad+ \frac{2}{8} \int_0^t \int_{\Omega} \lvert \nabla \varphi (s) \rvert^2 \; \d x \; \d s \bigg\}.
\end{align*}
By virtue of~\eqref{Eq: Energy of phi}, we can absorb the term involving $\nabla \varphi$ from the right-hand side to the left-hand side, delivering the estimate
\begin{align*}
 \| \varphi (t) \|_{L^2}^2 &+ \int_0^t \int_{\Omega} \lvert \nabla \varphi (s) \rvert^2 \; \d x \; \d s \\
 &\leq 4 \big\{ C(t_0 , q) + C \sup_{0 < s < t_0} \| \nabla [d(s) - b_{ n}] \|_{L^p}^2 + C(q , \| \nabla b_{ n} \|_{L^q}^2) \big\} \int_0^t \| \varphi (s) \|_{L^2}^2 \; \d s.
\end{align*}
Since $t \mapsto \| \varphi (t) \|_{L^2}^2$ is continuous on $[0 , T)$, Gronwall's inequality can be applied and reveals $\varphi \equiv 0$. \qed

\section{A digression on the weak Stokes operator and the proof of Theorem~\ref{Regularity}}
\label{Sec: Regularity}
Now, that we have constructed a mild solution to~\eqref{eq:PLCD} in the sense of~\eqref{eq:mild_solution}, we use the theory of maximal regularity, cf.~Section~\ref{Sec: Preliminaries}, in order to gain some additional regularity properties of the solutions. 
For this purpose, a suitable functional framework is needed. \par
In this section, let $p$ be such that $\lvert 1 / p - 1 / 2 \rvert < 1 / 6 + \eps$, and $p^{\prime}$ always denotes the H\"older conjugate exponent to $p$, i.e. $\tfrac{1}{p}+ \tfrac{1}{p^{\prime}}=1$.
Recall that $W^{- 1 , p}_{\sigma} (\Omega) = [W^{1,p^{\prime}}_{0,\sigma}(\Omega)]^*$ is defined as dual space, and denote the duality pairing by 
\begin{align*}
w(v) =\langle w,v \rangle_{W^{- 1 , p}_{\sigma},W^{1,p^{\prime}}_{0,\sigma}}, \quad  w\in W^{- 1 , p}_{\sigma} (\Omega), \, v \in W^{1,p^{\prime}}_{0,\sigma}(\Omega).
\end{align*}
Moreover, let 
$\Phi:~[L^{p^{\prime}}_{\sigma}(\Omega)]^{*} \rightarrow L^p_{\sigma}(\Omega)$
denote the canonical isomorphism between $[L^{p^{\prime}}_{\sigma}(\Omega)]^{*}$ and $L^p_{\sigma}(\Omega)$ introduced in Section~\ref{Sec: Preliminaries}, and the duality pairing is
\begin{align*}
(\Phi^{-1}u)(v)= \langle \Phi^{-1}u,v \rangle_{[L^{p^{\prime}}_{\sigma}]^{*}, L^{p^{\prime}}_{\sigma}} = \langle u,v \rangle_{L^p_{\sigma}, L^{p^{\prime}}_{\sigma}}= \int_ {\Omega} u \cdot \overline{v} \; \d x, \quad u \in  L^p_{\sigma}(\Omega), \, v\in L^{p^{\prime}}_{\sigma}(\Omega).
\end{align*} 
We regard $\Phi^{-1}$ also  as the canonical inclusion of $L^p_{\sigma} (\Omega)$ into $W^{-1 , p}_{\sigma} (\Omega)$ by 
\begin{align*}
\langle \Phi^{-1}u, v \rangle_{W^{- 1 , p}_{\sigma},W^{1,p^{\prime}}_{0,\sigma}}= \langle u,v \rangle_{L^p_{\sigma}, L^{p^{\prime}}_{\sigma}}, \quad u\in L^p_{\sigma} (\Omega), \, v\in W^{1,p^{\prime}}_{0,\sigma}(\Omega).
\end{align*} 
In this sense, we define the weak Stokes operator $\mathcal{A}_p$ in $W^{-1 , p}_{\sigma} (\Omega)$
by $\dom(\mathcal{A}_p) := \Phi^{-1} W^{1 , p}_{0 , \sigma} (\Omega)$ and 
\begin{equation}
\begin{aligned}
\mathcal{A}_p: \dom(\mathcal{A}_p) \subset W^{-1 , p}_{\sigma} (\Omega) \rightarrow W^{-1,p}_{\sigma}(\Omega), \quad w \mapsto \Big[v \mapsto \int_{\Omega} \nabla \Phi w \cdot \overline{ \nabla v } \; \d x \Big].
\end{aligned}
\label{reg1}
\end{equation}
Recall that, by Proposition~\ref{Prop: Square root properties}, the square root of the Stokes operator satisfies $A^{\frac{1}{2}}_{p^{\prime}} \in \text{Isom}(W^{1,p^{\prime}}_{0,\sigma}(\Omega),L^{p^{\prime}}_{\sigma}(\Omega))$ and hence
\begin{align}
\label{Eq: Square root isomorphism}
 \big[A^{\frac{1}{2}}_{p^{\prime}} \big]^{*}\Phi^{-1}  \in \text{Isom}(L^{p}_{\sigma}(\Omega),W^{-1,p}_{\sigma}(\Omega)).
\end{align}

We now show that the following representations of $\mathcal{A}_{p}$ are valid.
\begin{lemma}
\label{Lem: Representation of weak Stokes}
For $\lvert 1 / p - 1 / 2 \rvert < 1 / 6 + \eps$ the operator $\mathcal{A}_p$ is given by
\begin{align}
\label{Eq: Representation of weak Stokes}
\mathcal{A}_{p}= \big[A^{\frac{1}{2}}_{p'} \big]^* \Phi^{-1} \circ A_{p} \circ A^{-\frac{1}{2}}_{p} \Phi = \big[A^{\frac{1}{2}}_{p'} \big]^* \Phi^{-1} \circ A_{p} \circ \Phi \big[A^{-\frac{1}{2}}_{p'} \big]^*.
\end{align}
\end{lemma}

\begin{proof}
By definition, we find for any $u\in \dom(\mathcal{A}_p)$, i.e., $\Phi u\in W^{1 , p}_{0 , \sigma} (\Omega)$, and $v \in W^{1,p^{\prime}}_{0,\sigma}(\Omega)$ that
\begin{align*}
\big\langle [A^{\frac{1}{2}}_{p'}]^* \Phi^{-1} \circ A_{p} \circ A^{-\frac{1}{2}}_{p} \Phi u, v  \big\rangle_{W^{-1 , p}_{\sigma} , W^{1 , p^{\prime}}_{0 , \sigma}} &= \big\langle \Phi^{-1}  A^{\frac{1}{2}}_{p} \Phi u, A^{\frac{1}{2}}_{p'} v  \big\rangle_{[L^{p^{\prime}}_{\sigma}]^* , L^{p^{\prime}}_{\sigma}} \\
&= \int_{\Omega} A^{\frac{1}{2}}_{p} \Phi u \cdot \overline{A^{\frac{1}{2}}_{p'} v} \; \d x= \int_{\Omega} \nabla \Phi u \cdot \overline{\nabla v} \; \d x,
\end{align*}
where one verifies directly that $\dom(\big[A^{\frac{1}{2}}_{p'} \big]^* \Phi^{-1} \circ A_{p} \circ A^{-\frac{1}{2}}_{p} \Phi) =\dom(\mathcal{A}_p)$.
This proves the first identity.

To prove the second identity, 
notice that 
\begin{align}\label{eq:domAp}
\dom\left(\big[A^{\frac{1}{2}}_{p'} \big]^* \Phi^{-1} \circ A_{p} \circ \Phi \big[A^{-\frac{1}{2}}_{p'} \big]^*\right) = \{w\in W^{-1,p}_{\sigma}(\Omega) \mid \Phi \big[A^{-\frac{1}{2}}_{p'} \big]^*w \in \dom (A_p) \}.
\end{align}
To prove inclusions of domains suppose that $w \in \dom(\mathcal{A}_p)$. Then for any $v\in L^{p^{\prime}}_{\sigma} (\Omega)$
\begin{align}
\label{Eq: Identity for weak Stokes}
\begin{aligned}
 \big\langle \Phi \big[A^{-\frac{1}{2}}_{p'} \big]^*  w , v \big\rangle_{L^{ p}_{\sigma} , L^{p^{\prime}}_{\sigma}}
 &= \big\langle \big[A^{-\frac{1}{2}}_{p'} \big]^*  w , v \big\rangle_{[L^{p^{\prime}}_{\sigma}]^* , L^{p^{\prime}}_{\sigma}} =  \big\langle   w , A^{-\frac{1}{2}}_{p'} v \big\rangle_{W^{-1 , p}_{\sigma} , W^{1 , p^{\prime}}_{0 , \sigma}} \\
 &=
 \big\langle   \Phi w , A^{-\frac{1}{2}}_{p'} v \big\rangle_{L^{ p}_{\sigma} , L^{p^{\prime}}_{\sigma}}
 =  \big\langle  A^{-\frac{1}{2}}_{p} \Phi w ,  v \big\rangle_{L^{ p}_{\sigma} , L^{p^{\prime}}_{\sigma}},
\end{aligned}
\end{align}
and since $ A^{-\frac{1}{2}}_{p} \Phi w \in \dom (A_p)$, it follows that $w$ is in the set~\eqref{eq:domAp}.

The other way round, suppose that $w$ is in the set~\eqref{eq:domAp}. Then for any $v\in W^{1,p^{\prime}}_{0,\sigma}(\Omega)$
\begin{align*}
 \big\langle  w , v \big\rangle_{W^{-1 , p}_{\sigma} , W^{1 , p^{\prime}}_{0 , \sigma}}
 &=   \big\langle \big[A^{\frac{1}{2}}_{p'} \big]^* \big[A^{-\frac{1}{2}}_{p'} \big]^* w , v \big\rangle_{W^{-1 , p}_{\sigma} , W^{1 , p^{\prime}}_{0 , \sigma}} 
 = \big\langle \Phi \big[A^{-\frac{1}{2}}_{p'} \big]^* w , A^{\frac{1}{2}}_{p'}v \big\rangle_{L^{ p}_{\sigma} , L^{p^{\prime}}_{\sigma}} \\
&= \big\langle A^{\frac{1}{2}}_{p} \Phi \big[A^{-\frac{1}{2}}_{p'} \big]^* w , v \big\rangle_{L^{ p}_{\sigma} , L^{p^{\prime}}_{\sigma}} 
= \big\langle \Phi^{-1}A^{\frac{1}{2}}_{p} \Phi \big[A^{-\frac{1}{2}}_{p'} \big]^* w , v \big\rangle_{W^{-1 , p}_{\sigma} , W^{1 , p^{\prime}}_{0 , \sigma}}.
\end{align*}
By assumption $\Phi \big[A^{-\frac{1}{2}}_{p'} \big]^* w\in \dom (A_p)$, hence, setting $u = A^{\frac{1}{2}}_{p} \Phi \big[A^{-\frac{1}{2}}_{p'} \big]^* w$, one finds $u\in W^{1 , p}_{0, \sigma} (\Omega)$ with $\Phi u = w$, whence the equality of the domains follows. \par
Finally, the representation~\eqref{Eq: Representation of weak Stokes} follows from the identity~\eqref{Eq: Identity for weak Stokes} together with the first identity established in this proof.
\end{proof}

Since $\mathcal{A}_{p}$ is related to $A_p$ by a similarity transform, 
we can carry 
spectral properties of $A_p$ over to $\mathcal{A}_p$, and we obtain the following proposition as a corollary of Lemma~\ref{Lem: Representation of weak Stokes}, compare e.g.~\cite{DHP}. 

\begin{proposition}
\label{Prop: Analyticity of weak Stokes}
Let $\Omega \subset \R^3$ be a bounded Lipschitz domain. Then there exists $\eps > 0$ such that for $\lvert 1 / p - 1 / 2 \rvert < 1 / 6 + \eps$, it holds $\rho(A_p) = \rho(\mathcal{A}_p)$, 
\begin{itemize}
\item[(a)] $- \mathcal{A}_p$ generates a bounded analytic semigroup on $W^{-1 , p}_{\sigma} (\Omega)$, and for $u \in W^{-1 , p}_{\sigma} (\Omega)$ and $f \in L^p_{\sigma} (\Omega)$ the following two identities hold
\begin{align*}
 \mathrm{(1)} \quad e^{- t \mathcal{A}_p} u = \big[A^{\frac{1}{2}}_{p'}\big]^*\Phi^{-1}  e^{- t A_p} \Phi \big[ A_{p^{\prime}}^{- \frac{1}{2}} \big]^* u, \qquad \qquad \mathrm{(2)} \quad \Phi^{-1} e^{- t A_p} f = e^{- t \mathcal{A}_p} \Phi^{-1} f;
\end{align*}
\item[(b)] $\mathcal{A}_p$ has the maximal regularity property.
\end{itemize}
\end{proposition}

Now, we come to the proof of Theorem~\ref{Regularity}. 

\begin{proof}[Proof of Theorem~\ref{Regularity}]
Fix $3 \leq p < 3 + \eps$ and let $u$ and $d$ be mild solutions corresponding to Theorem~\ref{thm:main} on $[0 , T)$. In the following, we show that $\Phi^{-1} u$ and $d$ are mild solutions to the (weak) linear Stokes and heat equations with the respective right-hand sides
\begin{align*}
F_u= - \IP \mathrm{div} (u \otimes u + [\nabla d]^{\top} \nabla d) \quad \text{and} \quad F_d=- (u \cdot \nabla) d + \lvert \nabla d \rvert^2 d.
\end{align*}
Since for $d$ this has been proven in Subsection~\ref{Subsec: Retrieving}, we concentrate on $\Phi^{-1} u$. Recall that $e^{- t A_{p / 2}} \IP \divergence$, for $t>0$, \textit{a priori} defines a bounded operator on a dense subset of $L^{p/2}(\Omega)^{3\times 3}$ containing $C_c^{\infty} (\Omega)^{3 \times 3}$, and its closure defines a bounded operator from $L^p(\Omega)^{3 \times 3}$ to $L^p_{\sigma} (\Omega)$.

Furthermore, $\IP \divergence F$ for $F \in L^{p / 2}(\Omega)^{3 \times 3}$ is identified with an element in $W^{-1 , p}_{\sigma}(\Omega)$ by
\begin{align*}
 \big\langle  \IP \divergence F, v \big\rangle_{W^{-1 , p}_{\sigma} , W^{1 , p^{\prime}}_{0 , \sigma}}
 =\big\langle F , \nabla v \big\rangle_{L^{ p} , L^{p^{\prime}}}, \quad v \in W^{1 , p^{\prime}}_{0 , \sigma}(\Omega).
\end{align*}
Combining this together with Proposition~\ref{Prop: Analyticity of weak Stokes}, $\Phi^{-1} u$ satisfies
\begin{align*}
 \Phi^{-1} u(t) = e^{- t \mathcal{A}} \Phi^{-1} a - \int_0^t \e^{- (t - s) \mathcal{A}} \IP \divergence (u \otimes u + [\nabla d]^{\top} \nabla d) \; \d s.
\end{align*}
Now, since by Theorem~\ref{thm:main}
\begin{align*}
 t &\mapsto e^{\frac{\omega t}{2}} u(t) \in BC([0 , T) ; L^p_{\sigma} (\Omega)), \quad 
t \mapsto d(t) \in BC([0 , T) ; L^\infty (\Omega)^{3}), \\
 t &\mapsto e^{\frac{\omega t}{2}} \nabla d(t) \in BC([0 , T) ; L^p (\Omega)^{3\times 3})
\end{align*}
for some $\omega > 0$, we deduce that for all $s\in [1,\infty]$ we have
\begin{align*}
F_u&= - \IP \mathrm{div} (u \otimes u + [\nabla d]^{\top} \nabla d) \in L^{s} (0 , T ; W^{-1 , \frac{p}{2}}_{0 , \sigma} (\Omega)), \\
F_d&= - (u \cdot \nabla) d + \lvert \nabla d \rvert^2 d \in L^{s} (0 , T ; L^{\frac{p}{2}} (\Omega)^3).
\end{align*}
Now, if for some $1 < s < \infty$ the initial conditions $\Phi^{-1} a$ and $b$ satisfy
\begin{align}
\label{Eq: Condition for initial values}
 \Phi^{-1} a \in \big( W^{-1 , \frac{p}{2}}_{\sigma} (\Omega) , \Phi^{-1} W^{1 , \frac{p}{2}}_{0 , \sigma} (\Omega) \big)_{1 - \frac{1}{s} , s} \quad \text{and} \quad b \in (L^{\frac{p}{2}} (\Omega)^3 , \dom(B_{\frac{p}{2}}))_{1 - \frac{1}{s} , s},
\end{align}
then the maximal regularity of $\mathcal{A}_{\frac{p}{2}}$ (see Proposition~\ref{Prop: Analyticity of weak Stokes}) and $B_{\frac{p}{2}}$ (see Proposition~\ref{Prop: Maximal regularity of strong operators}) implies that $\Phi^{-1} u$ and $d$ satisfy
\begin{align}
\label{Eq: Maximal regularity classes for solutions}
\begin{aligned}
 &\Phi^{-1} u \in W^{1 , s} (0 , T ; W^{-1 , \frac{p}{2}}_{\sigma} (\Omega)) \cap L^s (0 , T ; \Phi^{-1} W^{1 , \frac{p}{2}}_{0 , \sigma} (\Omega)), \\
 & d^{\prime} , B_{\frac{p}{2}} d \in L^s (0 , T ; L^{\frac{p}{2}} (\Omega)^3)
\end{aligned}
\end{align}
and that they solve the respective equations~\eqref{eq:PLCD}
for almost every $0 < t < T$. Thus, in order to gain this regularity property, it remains to prove~\eqref{Eq: Condition for initial values}. For $b$, we have by Proposition~\ref{Prop: Square root properties} followed by~\cite[Cor.~6.6.3]{Haase}, and~\cite[p.~25]{Triebel1978}
\begin{align*}
 b \in W^{1 , p} (\Omega)^3 \subset W^{1 , \frac{p}{2}} (\Omega)^3 = \dom(B^{\frac{1}{2}}_{p / 2}) \subset \big(L^{\frac{p}{2}} (\Omega)^3 , \dom(B_{\frac{p}{2}}) \big)_{\frac{1}{2} , \infty} \subset \big(L^{\frac{p}{2}} (\Omega)^3 , \dom(B_{\frac{p}{2}}) \big)_{1 - \frac{1}{s} , s}
\end{align*}
for any $1 < s < 2$. For the weak Stokes operator a similar calculation works on the $L^{\frac{p}{2}}_{\sigma}$-scale instead of the $W^{1 , \frac{p}{2}}$-scale once we know that $\dom(\mathcal{A}_{p / 2}^{\frac{1}{2}}) = \Phi^{-1} L^{\frac{p}{2}}_{\sigma} (\Omega)$. Notice that since $0 \in \rho(A_{\frac{p}{2}})$ we find by Proposition~\ref{Prop: Analyticity of weak Stokes} that $0 \in \rho(\mathcal{A}_{\frac{p}{2}})$. Thus, by definition
\begin{align*}
 \dom(\mathcal{A}_{p / 2}^{\frac{1}{2}}) = \mathrm{Rg} (\mathcal{A}_{p / 2}^{- \frac{1}{2}})
\end{align*}
and the special structure of the similarity transform proven in Lemma~\ref{Lem: Representation of weak Stokes} shows that $$\dom(\mathcal{A}_{p / 2}^{\frac{1}{2}})= \Phi^{-1}(L^{\frac{p}{2}}_{\sigma}(\Omega))=[L^{(\frac{p}{2})^{\prime}}_{\sigma} (\Omega)]^*.$$
Finally, we deduce by the same reasoning as for the Laplacian that
\begin{align*}
 \Phi^{-1} a \in [L^{p^{\prime}}_{\sigma} (\Omega)]^* \subset [L^{(\frac{p}{2})^{\prime}}_{\sigma} (\Omega)]^* = \dom(\mathcal{A}_{p / 2}^{\frac{1}{2}}) &\subset \big(W^{-1 , \frac{p}{2}}_{\sigma} (\Omega) , \Phi^{-1} W^{1 , \frac{p}{2}}_{0 , \sigma} (\Omega) \big)_{\frac{1}{2} , \infty} \\
 &\subset \big(W^{-1 , \frac{p}{2}}_{\sigma} (\Omega) , \Phi^{-1} W^{1 , \frac{p}{2}}_{0 , \sigma} (\Omega) \big)_{1 - \frac{1}{s} , s},
\end{align*}
for any $1 < s < 2$. This proves~\eqref{Eq: Condition for initial values}.
\end{proof}

\section{Proofs of Theorems~\ref{Dirichlet-main} and~\ref{Dirichlet-regularity}}
\label{Sec: Dirichlet boundary conditions}
In this section, we discuss the existence, uniqueness, and regularity of mild solutions when the Dirichlet boundary data for the director field is a constant vector $e$. \par
Let us denote $\delta= d-e $. Then, the system~\eqref{eq:LCD} is equivalent to
\begin{equation}
\left\{
\begin{aligned}
\partial_{t} u + (u\cdot \nabla) u - \Delta u + \nabla \pi &= - \text{div}([\nabla \delta]^{\top} \nabla \delta) \quad &&\text{in} \ (0,T)\times \Omega,\\
\partial_{t} \delta - \Delta \delta + (u \cdot \nabla) \delta &= \vert \nabla \delta \vert^{2} \delta+ \vert \nabla \delta \vert^{2} e \quad &&\text{in} \ (0,T)\times \Omega, \\
\text{div}\ u &=0 \quad &&\text{in} \ (0,T)\times \Omega,\\
\lvert \delta + e \rvert &= 1 &&\text{in} \ (0 , T) \times \Omega, \\
(u, \delta) &= (0,0)  &&\text{on} \ (0,T)\times \partial \Omega, \\
(u,\delta)\Big\vert_{t=0} &= (a, \tilde{b})  &&\text{in}\ \Omega,
\end{aligned}
\right.
\label{model3}
\end{equation}
where $\tilde{b}=b-e $. We would like to emphasise that the system~\eqref{model3} in $(u,\delta)$ has homogeneous Dirichlet boundary conditions and $\nabla \delta=\nabla d $. \par
Dropping the condition $\lvert \delta + e \rvert = 1$ for a moment, we reformulate the problem  as 
\begin{align*}
\left\{
\begin{array}{rll}
\partial_t u   + A u  & = - \PP (u \cdot \nabla) u - \PP \div ([\nabla \delta]^{\top} \nabla \delta ), \quad &\text{ in } \ (0,T) \times \Omega,  \\
\partial_t \delta + B \delta  & = - (u \cdot \nabla) \delta +  \abs{\nabla \delta}^2 (\delta+e), \quad &\text{ in } \  (0,T) \times \Omega, \\
\end{array}\right.
\end{align*}
which defines a system in the space
\begin{align*}
L^q_{\sigma}(\Omega) \times L^q (\Omega)^3,
\end{align*}
where $B$ now denotes the negative Dirichlet Laplacian which is defined similarly to the Neumann Laplacian using the form
\begin{align*}
\fb : W_0^{1 , 2} (\Omega)^3 \times W_0^{1 , 2} (\Omega)^3 \to \IC, \quad &(u , v) \mapsto \int_{\Omega} \nabla u \cdot \overline{\nabla v} \; \d x.
\end{align*} 
Note that due to the availability of heat kernel estimates~\cite[Cor.~3.2.8]{Davies} and the validity of the square root property~\cite[Thm.~7.5]{Jerison_Kenig_Dirichlet} the counterparts of Propositions~\ref{Prop: Square root properties} and~\ref{prop:smoothing} are valid for the Dirichlet Laplacian. Especially, the $L^p$-$L^q$-estimates hold on all of $L^p(\Omega)^3$, i.e., average free spaces need not be considered. Furthermore, the maximal regularity of the negative Dirichlet Laplacian follows by~\cite[Cor.~1.1]{Lamberton}. \par
Denoting the nonlinear terms as 
\begin{align*}
F_u(u, \nabla \delta) &= - \PP \div (u \otimes u + [\nabla \delta]^{\top} \nabla \delta), \\
F_\delta(u, \nabla \delta, \delta)  &= - (u \cdot \nabla) \delta + \abs{\nabla \delta}^2 (\delta+ e ),
\end{align*}
we define the iteration scheme corresponding to the mild formulation~\eqref{eq:mildDirichlet} as follows. For $j\in \N_0$, define
\begin{eqnarray*}
u_0 :=e^{-tA}a,& & u_{j+1}  := u_0 + \int_0^t e^{-(t-s)A} F_u(u_j(s), \nabla \delta_j(s)) \; \d s,\\
\delta_{0} :=e^{-t B} \tilde{b},& & \delta_{j+1}  := \delta_0 + \int_0^t e^{-(t-s)B} F_\delta(u_j(s), \nabla \delta_j(s), \delta_j(s)) \; \d s.
\end{eqnarray*}
The analysis towards the proof of existence and uniqueness follows verbatim the proof for the case of Neumann boundary conditions since $\abs{e}=1 $ (which replaces the $\overline{b} $ in the previous case). Also note that in this case we do not need to split the equation for the director field as the Dirichlet Laplacian generates an exponentially stable semigroup on all of $L^q (\Omega)^3$. \par
Once the existence and uniqueness of $u$ and $\delta$ have been established, we can then return to the original variable $d=\delta+e $ and retrieve the condition $\abs{d}=1 $ by following the same arguments as in the previous case and by noting that $\vert d \vert^2-1=\vert e \vert^2 -1 =0 $ on the boundary $(0,T) \times \partial \Omega $. Finally, the discussion in Section~\ref{Sec: Regularity} stays literally the same.

\section{Concluding remarks}
\label{Sec: Comparison}
We would like to conclude by discussing other results and techniques.

First, in the case of a smooth domain $\Omega$ our approach yields similar results as has been obtained by Hieber \textit{et al.\@} in~\cite{Hieber_et_al} using quasilinear techniques. 
More concretely, the approach in~\cite{Hieber_et_al} requires initial data in Besov spaces
\begin{align*}
a\in B^{2\mu-2/p}_{qp}(\Omega)^3 \cap L^p_{\sigma}(\Omega), \quad
b\in B^{2\mu-2/p}_{qp}(\Omega)^3 , \quad \tfrac{2}{p}+\tfrac{3}{q}<1, \tfrac{1}{2}+ \tfrac{1}{p}+ \tfrac{3}{2q}<\mu\leq 1,
\end{align*}
using the fact that the embedding $B^{2\mu-2/p}_{qp}(\Omega)\hookrightarrow C^1(\overline{\Omega})$ holds. These initial data are much more regular than the ones assumed by us.

Recall that for $\Omega$ smooth, our results are valid for $3\leq p < q <\infty$, cf.\@ Remark~\ref{rem:smooth}. 
Now suppose that we have initial data $(a,d )\in W^{1,\frac{p}{2}}_{\sigma}(\Omega) \times W^{1,p}(\Omega)^3$ for some $p>9$. This choice ensures that $a$ and $d$ are bounded functions. 
Now we can repeat our arguments as before (see Section~\ref{Sec: Regularity}) to first obtain that $B_{\frac{p}{2}} d \in L^s(0,T; L^{\frac{p}{2}} (\Omega)^3)$ and then $u \in L^s(0,T; W^{1,\frac{p}{2}}_{0,\sigma}(\Omega))$ for $s \in (1,2)$. Since $\dom(B_{\frac{p}{2}})$ is a subspace of $W^{2 , \frac{p}{2}} (\Omega)^3$ for smooth domains, it is possible to control two derivatives of $d$. But now, since $\| u \|_{L^{p}_{\sigma}}$ and $\| \nabla d \|_{L^p}$ are also bounded in time, we observe that the right-hand side in the fluid equation is actually in $L^s(0,T; L^{\frac{p}{3}}_{\sigma}(\Omega)) $ and hence we infer that
\begin{align*}
 u \in W^{1 , s} (0 , T ; L^{\frac{p}{3}}_{\sigma} (\Omega)) \cap L^s (0 , T ; \dom(A_{\frac{p}{3}})), \qquad d^{\prime} , B_{\frac{p}{2}} d \in L^s (0 , T ; L^{\frac{p}{2}} (\Omega)^3)
\end{align*}
and thus $u$ and $d$ are strong solutions. Especially, $u(t) \in \dom(A_{\frac{p}{3}})$ and $d(t) \in \dom(B_{\frac{p}{2}})$ for almost every $t \in (0 , T)$ so that both lie in $C^1(\overline{\Omega})$ in the spatial variables as used in~\cite{Hieber_et_al}.

Note that there are also other versions of the simplified Ericksen--Leslie model. For instance, some authors, see, e.g.,~\cite{Lin_Liu, HuWang2010} and the references therein, drop the assumption $\abs{d}=1$ and replace the dynamical equation for the  director field $d$ by
\begin{align*}
\partial_{t} d - \Delta d +(u \cdot \nabla) d =  -\gamma f(d), \quad \gamma>0,
\end{align*}
for a bounded vector valued penalty function $f$. In particular, Hu and Wang considered in~\cite{HuWang2010} the case $f=0$.  The method we presented here can be adapted for this setting as well.

Finally, we would like to remark that our approach, based on the iteration scheme, has been crucially based upon the fact that the right-hand side (nonlinearities) in the fluid equation can be written in a divergence form. Since the same remains true for more general models arising in nematic liquid crystals, we are hopeful that this method shall turn out to be fruitful even in such general situations.

\subsection*{Acknowledgement} 
We would like to thank Matthias Hieber for introducing us to this interesting field of research.

\begin{thebibliography}{40}


\bibitem{ABHN}
W.~Arendt, C.~J.~K.~Batty, M.~Hieber, and F.~Neubrander.
\newblock Vector-valued Laplace Transforms and Cauchy Problems.
\newblock Monographs in Mathematics, vol.~96, Birkh\"auser, Basel-Boston-Berlin, 2001.
\newblock \doi{10.1007/978-3-0348-5075-9}


\bibitem{Dahlberg}
B.~E.~J.~Dahlberg.
\newblock {\em $L^q$-estimates for Green potentials in Lipschitz domains.}
\newblock Math.~Scand.~\textbf{44} (1979), no.~1, 149--170.
\newblock \doi{10.7146/math.scand.a-11800}


\bibitem{Davies}
E.~B.~Davies.
\newblock Heat kernels and spectral theory.
\newblock Cambridge Tracts in Mathematics, 92,
\newblock Cambridge University Press, Cambridge, 1989.
\newblock \doi{10.1017/CBO9780511566158}


\bibitem{DHP}
R.~Denk, M.~Hieber, and K.~Pr\"uss.
\newblock {\em $\mathcal{R}$-boundedness, Fourier multipliers and problems of elliptic and parabolic type.}
\newblock Mem.~Amer.~Math.~Soc.~\textbf{166} (2003), no.~788.
\newblock \doi{10.1090/memo/0788}


\bibitem{Deuring}
P.~Deuring.
\newblock {\em The Stokes resolvent in 3D domains with conical boundary points: nonregularity in $L^p$-spaces.}
\newblock Adv.~Differential Equations \textbf{6} (2001), no.~2, 175--228.
\newblock \url{https://projecteuclid.org/euclid.ade/1357141493}



\bibitem{Ericksen}
J.~L.~Ericksen.
\newblock {\em Hydrostatic theory of liquid crystals.}
\newblock Arch.~Ration.~Mech.~Anal.~\textbf{9} (1962), 371--378.
\newblock \doi{10.1007/BF00253358}


\bibitem{Fabes_Mendez_Mitrea}
E.~Fabes, O.~Mendez, and M.~Mitrea.
\newblock {\em Boundary layers on Sobolev-Besov spaces and Poisson's equation for the Laplacian in Lipschitz domains.}
\newblock J.~Funct.~Anal.~\textbf{159} (1998), no.~2, 323--368.
\newblock \doi{10.1006/jfan.1998.3316}


\bibitem{Geissert_et_al}
M.~Geissert, M.~Hess, M.~Hieber, C.~Schwarz, and K.~Stavrakidis.
\newblock {\em Maximal $L^p - L^q$-Estimates for the Stokes Equation: a Short Proof of Solonnikov's Theorem.}
\newblock J.~math.~fluid mech.~\textbf{12} (2010), no.~1, 47--60.
\newblock \doi{10.1007/s00021-008-0275-0}

\bibitem{Giga_Analyticity}
Y.~Giga.
\newblock {\em Analyticity of the semigroup generated by the Stokes operator in $L_r$ spaces}.
\newblock Math.~Z.~\textbf{178} (1981), no.~3, 297--329.
\newblock \doi{10.1007/BF01214869}

\bibitem{Giga_fractional}
Y.~Giga.
\newblock {\em Domains of fractional powers of the Stokes operator in $L_r$ spaces}.
\newblock Arch.~Ration.~Mech.~Anal.~\textbf{89} (1985), no.~3, 251--265.
\newblock \doi{10.1007/BF00276874}

\bibitem{Giga}
Y.~Giga.
\newblock {\em Solutions for semilinear parabolic equations in $L^p$ and regularity of weak solutions of the Navier-Stokes system.}
\newblock J.~Differential Equations \textbf{62} (1986), no.~2, 186--212.
\newblock \doi{10.1016/0022-0396(86)90096-3}


\bibitem{GigaMiyakawa1985}
Y.~Giga and T.~Miyakawa.
\newblock {\em Solutions in $L_r$ of the Navier-Stokes initial value problem.}
\newblock Arch.~Ration.~Mech.~Anal.~\textbf{89} (1985), no.~3, 267--281.
.\newblock \doi{10.1007/BF00276875}

\bibitem{Haase}
M.~Haase.
\newblock The functional calculus for sectorial operators.
\newblock Operator Theory: Advances and Applications, 169. Birkh\"auser Verlag, Basel, 2006.
\newblock \doi{10.1007/3-7643-7698-8}

\bibitem{Hieber_et_al}
M.~Hieber, M.~Nesensohn, J.~Pr\"uss, and K.~Schade.
\newblock {\em Dynamics of nematic liquid crystal flows: the quasilinear approach.}
\newblock Ann.~Inst.~H.~Poincar\'e Anal.~Non Lin\'eaire \textbf{33} (2016), no.~2, 397--408.
\newblock \doi{10.1016/j.anihpc.2014.11.001}

\bibitem{HieberPruess2016}
M.~Hieber and J.~Pr\"uss.
\newblock {\em {M}odeling and {A}nalysis of the {E}ricksen-{L}eslie {E}quations for {N}ematic {L}iquid {C}rystal {F}lows.}
\newblock In {\em Handbook of Mathematical Analysis in Mechanics of Viscous Fluids}. Springer International Publishing, [Cham], 2016, 1--60.
\newblock \doi{10.1007/978-3-319-10151-4_26-1}


\bibitem{Hineman2014}
J.~L.~Hineman.
\newblock {\em A survey of results and open problems for the hydrodynamic flow of nematic liquid crystals.}
\newblock Electron. J. Differ. Equ. Conf.~\textbf{21} (2014), 149--172.
\newblock \url{https://ejde.math.txstate.edu/conf-proc/21/h1/hineman.pdf}

		
\bibitem{Hineman2013}
J.~L.~Hineman and C.~Wang.
\newblock {\em Well-posedness of nematic liquid crystal flow in {$L^3_{\rm uloc}(\mathbb{R}^3)$}.}
\newblock Arch.~Ration.~Mech.~Anal.~\textbf{210} (2013), no.~1, 177--218.
\newblock \doi{10.1007/s00205-013-0643-7}

\bibitem{HuWang2010}
X.~Hu and D.~Wang.
\newblock {\em Global solution to the three-dimensional incompressible flow of liquid crystals.}
\newblock Comm. Math. Phys. \textbf{296} (2016), no.~3, 861--880.
\newblock \doi{10.1007/s00220-010-1017-8}

\bibitem{Jerison_Kenig_Dirichlet}
D.~Jerison and C.~E.~Kenig.
\newblock {\em The inhomogeneous Dirichlet problem in Lipschitz domains.}
\newblock J.~Funct.~Anal.~\textbf{130} (1995), no.~1, 161--219.
\newblock \doi{10.1006/jfan.1995.1067}

\bibitem{Jerison-Kenig}
D.~Jerison and C.~E.~Kenig.
\newblock The functional calculus for the {L}aplacian on {L}ipschitz domains.
\newblock{\em Journees "Equations aux Derivees Partielles'' (Saint Jean de Monts, 1989),}
\newblock Exp. No. IV, 10 pp., Ecole Polytech., Palaiseau, 1989.
\newblock \url{http://www.math.sciences.univ-nantes.fr/~sjm/CDROM/data/pdf/1989/A5.pdf}


\bibitem{Kato}
T.~Kato.
\newblock {\em Strong $L^p$-solutions of the Navier-Stokes equations in $\R^m$, with applications to weak solutions.}
\newblock Math.~Z.~\textbf{187} (1984), no.~4, 471--480.
\newblock \doi{10.1007/BF01174182}


\bibitem{Kunstmann_Weis}
P.~Kunstmann and L.~Weis.
\newblock {\em New criteria for the $H^{\infty}$-calculus and the Stokes operator on bounded Lipschitz domains.}
\newblock J.~Evol.~Equ.~\textbf{17} (2017), no.~1, 387--409.
\newblock \doi{10.1007/s00028-016-0360-4}

\bibitem{Lamberton}
D.~Lamberton.
\newblock {\em Equations d'\'evolution lin\'eaires associ\'ees \`a des semi-groupes de contractions dans les espaces $L^p$.}
\newblock J.~Funct.~Anal.~\textbf{72} (1987), no.~2, 252--262. 

\bibitem{Leslie}
F.~M.~Leslie.
\newblock {\em Some constitutive equations for liquid crystals.}
\newblock Arch.~Ration.~Mech.~Anal.~\textbf{28} (1968), no.~4, 265--283.
\newblock \doi{10.1007/BF00251810}

\bibitem{Li}
J.~Li.
\newblock {\em Global strong solutions to the inhomogeneous incompressible nematic liquid crystal flow.}
\newblock Methods Appl.~Anal.~\textbf{22} (2015), no.~2, 201--220.
\newblock \doi{10.4310/MAA.2015.v22.n2.a4}

\bibitem{LinDing2012}
J.~Lin and Sh.~Ding.
\newblock {\em On the well-posedness for the heat flow of harmonic maps and the hydrodynamic flow of nematic liquid crystals in critical spaces.}
\newblock Math.~Methods~Appl.~Sci.~\textbf{35} (2012), no.~2, 158--173.
\newblock \doi{10.1002/mma.1548}

\bibitem{Li_Wang}
X.~Li and D.~Wang.
\newblock {\em Global solution to the incompressible flow of liquid crystals.}
\newblock J.~Differential Equations \textbf{252} (2012), no.~1, 745--767.
\newblock \doi{10.1016/j.jde.2011.08.045}

\bibitem{Lin_Lin_Wang}
F.-H.~Lin, J.~Lin, and C.~Wang.
\newblock {\em Liquid Crystal Flows in Two Dimensions.}
\newblock Arch.~Ration.~Mech.~Anal.~\textbf{197} (2010), no.~1, 297--336.
\newblock \doi{10.1007/s00205-009-0278-x}

\bibitem{Lin_Liu}
F.-H.~Lin and C.~Liu.
\newblock {\em Nonparabolic dissipative systems modeling the flow of liquid crystals.}
\newblock Comm.~Pure Appl.~Math. \textbf{48} (1995), no.~5, 501--537.
\newblock \doi{10.1002/cpa.3160480503}

\bibitem{PSZ}
J.~Pr\"uss, G.~Simonett, and R.~Zacher.
\newblock {\em On convergence of solutions to equilibria for quasilinear parabolic problems.}
\newblock J. Differential Equations,
\newblock \textbf{246} (2009), no. 10, 3902--3931.
\newblock \doi{10.1016/j.jde.2008.10.034}

\bibitem{Shen}
Z.~Shen.
\newblock {\em Resolvent estimates in $L^p$ for the Stokes operator in Lipschitz domains.}
\newblock Arch.~Ration.~Mech.~Anal.~\textbf{205} (2012), no.~2, 395--424.
\newblock \doi{10.1007/s00205-012-0506-7}


\bibitem{Virga2012}
S.~Sonnet and E.~G.~Virga.
\newblock Dissipative ordered fluids. Theories for liquid crystals.
\newblock Springer, New York, 2010.
\newblock \doi{10.1007/978-0-387-87815-7}

\bibitem{Tolksdorf_Dissertation}
P.~Tolksdorf.
\newblock On the $\mathrm{L}^p$-theory of the Navier-Stokes equations on Lipschitz domains.
\newblock Technische Universit\"at, Darmstadt, 2017.
\newblock \url{http://tuprints.ulb.tu-darmstadt.de/5960/}

\bibitem{Tolksdorf}
P.~Tolksdorf.
\newblock {\em On the $\mathrm{L}^p$-theory of the Navier--Stokes equations on three-dimensional bounded Lipschitz domains,}
\newblock \href{https://arxiv.org/abs/1703.01091}{arxiv:1703.01091}

\bibitem{Tolksdorf_Higher-order}
P.~Tolksdorf.
\newblock {\em $\mathcal{R}$-sectoriality of higher-order elliptic systems on general bounded domains,}
\newblock \doi{10.1007/s00028-017-0403-5}, to appear in J.~Evol.~Equ.

\bibitem{Triebel1978}
H.~Triebel.
\newblock Interpolation theory, function spaces, differential operators. Second Edition.
\newblock Johann Ambrosius Barth, Heidelberg 1978.
\newblock


\bibitem{Virga1994}
E.~G.~Virga.
\newblock Variational theories for liquid crystals.
\newblock Applied Mathematics and Mathematical Computation, 8. Chapman \& Hall, London, 1994.
\newblock \doi{10.1007/978-1-4899-2867-2}











\end{thebibliography}
\end{document}